%% file: NoetherianArxivI.tex
\documentclass[reqno]{amsart}
\usepackage{amsthm}
\usepackage{mathrsfs}
\usepackage{mathbbol}
\usepackage{ifthen}

\newcommand{\ds}[1]{\ {#1} \ }
\newcommand{\dss}[1]{\quad {#1} \quad }

\newcommand\thmref[2]{\pSkip\textbf{Theorem #1. }\emph{#2}\pSkip}
\newcommand\corref[2]{\pSkip\textbf{Corollary #1. }\emph{#2}\pSkip}

\def\RR{\mathscr R}
\def\corn{{\operatorname{corn}}}

\input tropMacro.tex
\def\tlv{\widetilde v }
\def\lclo{$\lmodnu$}

\def\mcR{\mathcal R}
\def\htR{\widehat R}

\def\mcI{\mathcal I}

\newcommand{\xl}[2]{\,\,{^{[#2]}}{#1}\,}

\def\lv{\operatorname{s}}

\def\2{\mathbb 2}

\def\lft{^{\rightarrow}}
\def\rgt{^{\leftarrow}}

\def\fgin{p} 
\def\fgot{q} 

\def\qq{u}

\def\pipe{{\underset{{\tG}}{\mid}}}

\def\lmod{\mathrel   \pipeGS \joinrel\joinrel \joinrel =}

\def\pipeWL{{\underset{L}{\mid}}}
\def\pipeL{{\underset{L}{\mid}}}

\def\lmodWL{\mathrel  \pipeWL   \joinrel \joinrel =}
\def\lmodWLnu{\mathrel  \pipeWL   \joinrel \joinrel \equiv_\nu}
\def\lmodLnu{\mathrel  \pipeL   \joinrel \joinrel \equiv_\nu}
\def\lmodnu{\mathrel  \mid   \joinrel   \equiv_\nu}

\def\rat{{\operatorname{rat}}}
\def\corn{{\operatorname{corn}}}

\def\tng{{\operatorname{tng}}}
\def\semiring0{semiring$^{\dagger}$}
\def\semirings0{semirings$^{\dagger}$}
\def\semifield0{semifield$^{\dagger}$}
\def\semifields0{semifields$^{\dagger}$}
\def\domain0{domain$^{\dagger}$}
\def\predomain0{pre-domain$^{\dagger}$}
\def\domains0{domains$^{\dagger}$}
\def\fone{{\one_F}}

\def\fzero{{\zero_F}}

\def\eBin{binomial}
\def\Z{\mathbb Z}
\def\Q{\mathbb Q}

\def\lmod{\mathrel  \pipe \joinrel \joinrel =}
\def\pipe{{\underset{{\tG}}{\mid}}}

\def\pipeL{{\underset{{L}}{\mid}}}

\def\pipe1{{\underset{{1}}{\mid}}}
\def\lmodL{\mathrel  \pipeL \joinrel \joinrel =}

\def\lmod1{\mathrel  \pipe1 \joinrel \joinrel =}

\def\corner{corner}

\def\tSS{S}

\def\pSkip{\vskip 1.5mm \noindent}

\def\gnu{>_{\nu}}
\def\nuge{\ge_{\nu}}
\def\mcomp{component}

\def\nucong{\cong_{\nu}}
\def\nule{\le_{\nu}}

\def\stak(f){\got{P}}
\def\stak{\got{S}}

\def\eroman{\etype{\roman}}
\def\bfa{\bold a}
\def\bfb{\bold b}
\def\bfc{\bold c}
\def\bfk{\bold k}

\newtheorem{thm}[theorem]{Theorem}
\newtheorem*{thm*}{Theorem}
\newtheorem*{dig*}{Digression}

\newtheorem{cor}[theorem]{Corollary}

\newtheorem{lem}[theorem]{Lemma}
\newtheorem{rem}[theorem]{Remark}
\def\pSkip{\vskip 1.5mm \noindent}

\newtheorem{prop*}{Proposition}

\newtheorem{prop}[theorem]{Proposition}
\newtheorem{defn}[theorem]{Definition}
\newtheorem*{examp*}{Example}
\newtheorem*{examples*}{Examples}
\newtheorem*{remark*}{Remark}
\newtheorem*{defn*}{Definition}
\def\ef{f^{\operatorname{es}}}

\def\bfi{ \textbf{i}}
\def\bfj{\textbf{j}}
\def\la{\lambda}

\def\pl{[} \def\pr{]}

\def\La{\Lambda}

\def\FunSR{\operatorname{Fun} (S,R)}
\def\FunSF{\operatorname{Fun} (S,F)}

\def\Fun{\operatorname{Fun}}
\def\Pol{\operatorname{Pol} }

\def\PolSF{\operatorname{Pol} (S,F)}
\def\tT{\mathcal T}
\def\tR{\mathcal R}

\numberwithin{equation}{section}

\def\M0{M_{\zero}}
\def\rzero{ \zero_R}

\def\aIdeal{exchange ideal}
\def\bIdeal{m-exchange ideal}
\def\supp{\operatorname{supp}}
\def\csupp{\operatorname{csupp}}
\def\tsupp{\operatorname{tsupp}}

\def\SR{R}

\def\PS{P}

\def\a{\alpha}

\def\rzero{\zero_\SR}

\def\rone{\one_\SR}

\newcommand{\etype}[1]{\renewcommand{\labelenumi}{(#1{enumi})}}

\newcommand{\nPS}[1]{\PS_{(!#1)}}
\newcommand{\nPSo}[1]{\nPS{\one}}

\newcommand{\cl}[1]{\bar{{#1}}}
\def\clF{\cl{F}}

\def\vmap{\vartheta}

\begin{document}


\title[Ideals of polynomial semirings in tropical c] {Ideals of polynomial semirings in tropical mathematics}

\author{Zur Izhakian}\thanks{This work has been supported by the
Israel Science Foundation grant No. 448/09.}
\address{Department of Mathematics, Bar-Ilan University, Ramat-Gan 52900,
Israel} \email{zzur@math.biu.ac.il}

\author{Louis Rowen}
\address{Department of Mathematics, Bar-Ilan University, Ramat-Gan 52900,
Israel} \email{rowen@macs.biu.ac.il}


\date{\today}



\begin{abstract} We describe the ideals, especially the  prime ideals, of semirings of polynomials over layered domains,
and in particular over  supertropical domains.  Since there are so
many of them, special attention is paid to the ideals arising from
layered varieties, for which we prove that every prime ideal is a
consequence of finitely many binomials. We also obtain layered
tropical versions of the classical Principal Ideal Theorem and
Hilbert Basis Theorem.
\end{abstract}

\maketitle


\numberwithin{equation}{section}


\section{Introduction}\label{sec:Introduction}

One classical technique of algebraic geometry is to exploit the
1:1 correspondence between varieties and ideals of the polynomial
algebra, thereby enabling one to transfer algebraic properties of
the coordinate algebra to the geometric properties of the variety.
This is also one of our goals in studying supertropical algebras
and structures.

 Tropical varieties
(cf.~\cite{MS09}, for example) can be obtained by taking the
``tropicalization'' of the common roots of polynomial ideals over
the field of Puiseux series. To cope with various structural
shortcomings of the max-plus algebra, supertropical semirings were
introduced as \emph{extended tropical arithmetic} in
\cite{zur05TropicalAlgebra}, and studied in greater depth in
\cite{IzhakianRowen2007SuperTropical}, in which radical ideals are
utilized to obtain a version of Hilbert's Nullstellensatz. The
supertropical structure was further generalized in
 \cite{IzhakianKnebuschRowen2009Refined}, to the
\textbf{layered \domain0}, having multiple ghost layers and
possibly no zero element. Thus, the supertropical \domain0 is a
special case, which for emphasis we call the \textbf{standard}
supertropical case. Likewise, the max-plus algebra is an instance
of what we call the \textbf{standard} tropical case.  The passage
from polynomials over the field of Puiseux series to supertropical
polynomials is described as the \textbf{tropicalization functor}
in \cite{IzhakianKnebuschRowen2010Categories}.

Given an algebraic structure, one is led to consider its ideal
theory. The prospects for success for semirings is clouded by the
failure of ideals to determine homomorphisms, and in the context
of universal algebra it is preferable to consider congruences.
Nevertheless theorems such as the supertropical Nullstellensatz
\cite[Theorem~7.17]{IzhakianRowen2007SuperTropical} indicate that
ideals should play a significant role in the theory, and
furthermore the prime (monoid) ideal spectrum  is featured in the
treatment of algebraic geometry over monoids in~\cite{CHWW}. In
this paper, our overall aim is to lay the rudiments of the
foundation for the intrinsic ideal theory of polynomial
\semirings0 over
 layered 1-\semifields0 (especially over
supertropical semifields). Even in one indeterminate, the theory
is considerably more complicated than the classical situation.
From the outset,  different polynomials may define the same
function; one polynomial might be reducible even while the other
is irreducible. This leads us to formulate polynomials rather as
elementary formulas in the appropriate language in model theory,
which also enables us to handle such important variants as Laurent
polynomials and rational polynomials, as explained in
\S\ref{model}. But  we still write polynomials in the familiar
notation, as sums of monomials.

 In view of the
Nullstellensatz(\cite[Theorem~7.17]{IzhakianRowen2007SuperTropical}),
  translated to the layered theory in
\cite[Theorem~6.13]{IzhakianKnebuschRowen2009Refined},
 we should like to classify the prime ideals of the
 polynomial \semiring0   over a
1-\semifield0. In \S\ref{primide}  we describe prime ideals of the
polynomial \semiring0 in one indeterminate over  a 1-\semifield0.
Nevertheless, there are far too many ideals of the polynomial
\semirings0 for a workable theory;  when considering all of them,
we are confronted with counterexamples to the major theorems from
classical Noetherian theory. Even the prime ideals are troublesome
to study.
 We deal with this difficulty by
 limiting the class of ideals under consideration to those most directly
 associated with tropical geometry. There are two competing ways of obtaining such ideals.

   We can focus on ideals of
polynomials defined in terms of their  corner loci, which we
call \textbf{corner ideals}. The prime corner ideals~$P$
satisfy a property which we call the m-\textbf{exchange} property,
and they also are $\nu$-\textbf{prime} in the sense that if $ab
\in P$ then $a$ or $b$ is $\nu$-equivalent to an element of $P$.
Over several indeterminates, the ``tropical'' world then satisfies
a new property not found in classical ideal theory:
\corref
 {\ref{binom2}}{Every $\nu$-prime ``\bIdeal'' of
$F[\la_1, \dots, \la_n]$ contains a binomial.}
A useful computational result, especially for prime ideals:
\thmref{\ref{exchange15}}
{Suppose for polynomials $f_1,\ f_2,\ h_1,$ and
$h_2$ that
 $f_1+h_1$  and $f_2+h_2$ are  in an m-exchange ideal $A
\triangleleft \mathcal R$, with $\supp(f_1)\cap \supp(f_2) =
\emptyset$, and furthermore that  $h_1(\bfa) \nucong  h_2(\bfa)$
for each corner root $\bfa$ of $A$. Then $h_1h_2(f_1+f_2) \in  A .$}%
Alternatively, one may study roots in terms of layers, especially
the tangible layer (1-layer), rather than in terms of corners, and
have variants of these results such as Theorem~\ref{exchange}.

Corollary~\ref{binom2} can be coupled with a general fact:
\thmref{\ref{genbin}}{ Any set of tangible binomials is generated
(together with the exchange property) of $ F[\Lm , \Lm ^{-1}]$
(for $\Lm  = \{ \la_1, \dots, \la _n \}$)
 by at most $n$ irredundant binomials,}  
 thereby yielding:
 \thmref{\ref{ngener0}}{The set of tangibly spanned binomials of
 any \bIdeal\ is generated (via the exchange property)
by at most $n$ irredundant binomials of $A$.}
 In the standard supertropical case, the
result is made more explicit in Theorem~\ref{ngener}.

For non-exchange ideals,
  we restrict our attention further, to  consider only
those ideals of supertropical polynomials that arise naturally in
the tropicalization functor of
\cite{IzhakianKnebuschRowen2010Categories}. In this case it
becomes  clear that each ``tropicalized'' ideal is finitely
generated, since it is generated by the tropicalization of a
Groebner-Shirshov base. The difficulty with this approach is that
the tropicalization functor does not respect the lattice structure
of ideals, cf.~Example~\ref{badtrop}. To obtain a more
comprehensive theory that is also self-contained, we axiomatize
the notion of tropicalization in Definition~\ref{tropc1} so that
it can work more easily, utilizing the philosophy of the
Groebner-Shirshov base. This enables us to prove versions of the
Principal Ideal Theorem and the Hilbert Basis Theorem. (These
results are really facts about the standard supertropical case,
since they deal with the tangible layer.)



\section{Review of the basic definitions}\label{chap:TropicalArithmetic}

We review the basic set-up, for the reader's convenience, taken
mostly from \cite{IzhakianRowen2007SuperTropical}. We recall that
a \textbf{semiring  without 0}, denoted in this paper as a
\textbf{\semiring0} $R := (R,+,\cdot\; ,\rone)$, is a set $R$
equipped with two binary operations $+$ and~$\cdot \;$, called
addition and multiplication, such that:
\begin{enumerate}
    \item $(\SR, +)$ is an Abelian semigroup; \pSkip
    \item $(\SR, \cdot \ , \rone )$ is a monoid with identity element
    $\rone$; \pSkip
    \item Multiplication distributes over addition. \pSkip
\end{enumerate}

A semiring is then a \semiring0\ with an additive identity element
$\rzero$ satisfying  \begin{equation}\label{with0} \rzero \cdot a
= a \cdot \rzero
 =\rzero, \quad \forall a\in R. \end{equation}

\begin{rem}\label{rem0} We prefer  to work with \semirings0  since they provide greater flexibility,
but there is not much difference between the theories of
\semirings0 and semirings. One can formally adjoin the element
~$\rzero$ to a \semiring0\ $R$ to obtain the semiring
$R_{\zero}:=R \cup \{\rzero\}$ whose multiplication is that of
$R$, where we also stipulate multiplication by~$\rzero$ according
to Equation~\eqref{with0}.\end{rem} One defines an \textbf{ideal}
$A$ of a commutative \semiring0 $R$ (written $A \triangleleft R$)
in the usual way: $A$ is an additive sub-semigroup such that $ ra
\in A$ for all $a\in A$ and $r\in R.$  An ideal $P$ of $R$ is a
\textbf{prime ideal} if it satisfies the usual condition that
$R\setminus P$ is a monoid under multiplication:
$$\text{$a,b \notin P$ \ds{implies} $ab \notin P$.}$$

\medskip


 \subsection{Layered \domains0}

Since the results of this paper are given n the framework of
uniform layered \domains0, as exposed in
\cite{IzhakianKnebuschRowen2009Refined}, let us review the main
example from
\cite[Construction~3.2]{IzhakianKnebuschRowen2009Refined}.
Throughout $L$ is a \semiring0 with unit element $1 := 1_L$. For
convenience, we assume throughout that $L = L_{\ge 1}$; i.e., $1$
is the minimal element of $L$.

\begin{example}\label{exs}

Suppose $\tT$ is a cancellative ordered multiplicative monoid,
viewed as a \semiring0\ in which  addition is given according to
the order of $\tT$, i.e., by $a+b = \max\{a,b\}$,
~\cite[Remark~3.1]{IzhakianKnebuschRowen2009Refined}. We define
the \textbf{uniform $L$-layered \domain0} $R:= \RR(L,\tT)$ to be
set-theoretically $L\times \tT$, where for $k,\ell\in L,$ and
$a,b\in\tT,$  we write $\xl{a }{k} $ for $(k,a)$ and  define
multiplication componentwise, i.e.,
\begin{equation}\label{13}
\xl{a }{k} \cdot \xl{b}{\ell} = \xl{(ab)}{k\ell},  \end{equation}
and addition from the rules:
\begin{equation}\label{14}
\xl{a }{k}+ \xl{b}{\ell}=\begin{cases} \xl{a }{k}& \quad\text{if}\quad a >  b,\\
\xl{b}{\ell}& \quad\text{if}\quad  a <  b,\\
\xl{a }{k+\ell}& \quad\text{if}\quad a=
b.\end{cases}\end{equation}
We write $R_k$ for the subset $\{\xl{a }{k}: a \in \tT\}$. The
``transition maps''
$$ \nu_{\ell,k}: R_k \to R_\ell, \qquad k \le \ell \in L,$$ are
given by $\xl{a}{k} \mapsto \xl{a}{\ell}$.

Note that $R_1$ is a multiplicative monoid isomorphic to $\tT$,
called the monoid of \textbf{tangible elements}.  Thus, $R_1$ can
be endowed with the given order of $\tT.$ We call $R$ a
1-\textbf{\semifield0} if $R_1$ is an Abelian group.
 We also define
$e_k :=  \xl{(\rone)}{k}.$ Then any element $a$ of $R_k$ can be
written uniquely in the form $a = e_k a_1$ for some  $a_1 \in
R_1.$ Likewise, for  $b = e_\ell b_1,$   we write $a\nucong b$ if
$a_1 = b_1$, and $a >_\nu b$ if $a_1>b_1$ (in $R_1$).

$ R:= \RR(L,\tT)$ is   equipped  with the sort map $\lv:  R \to L$
    given by $\lv( \xl{a }{k})=k.$ Thus, $\lv(ab) =
    \lv(a) \lv(b),$ and $\lv(a+b) \ge \max \{\lv(a), \lv(b)\}.$
\end{example}

Example~\ref{exs} was formalized in
\cite[Definition~3.6]{IzhakianKnebuschRowen2009Refined} and
generalized in
\cite[Definition~3.25]{IzhakianKnebuschRowen2009Refined}, but
since our interest in this paper is in the ideals of the
polynomial \semiring0 over Example~\ref{exs}, we do not bother
with these abstract definitions.

Given $\ell \in L,$ we say that $k \in L$ is $\ell$-\textbf{ghost}
if $k = \ell + p$ for some $p \in L$.  Note that $\ell$ itself can
be $\ell$-ghost if $\ell = \ell + p.$ If $\ell$ is $\ell$-ghost,
we call $\ell$ \textbf{infinite}; otherwise $\ell$ is called
\textbf{finite}. We write $L_{> \ell}$ (resp.~$L_{\ge \ell}$) for
the subset $\{k \in L: k> \ell \} \subset L$ (resp. $\{k: k \geq
\ell \} \subseteq L$).

\begin{example}\label{exs1} $ $

\begin{enumerate} \item When $L = \{ 1 \}$ we  have the max-plus algebra   studied in the usual tropical literature. In this case 1 is  1-ghost, but otherwise we always assume
that 1 is not  1-ghost.

\item When $L = \{ 1 , \infty\}$ we   have the ``standard''
supertropical situation studied in
\cite{IzhakianRowen2007SuperTropical}, where $\tT = R_1$ and $\tG
= R_\infty.$   In this case 1 is not 1-ghost, but $\infty$ is
ghost with respect to both  indices.

\end{enumerate}
\end{example}

 Most of the examples in this
paper are presented for the extended supertropical \semiring0
denoted as $D(\Real) = (\Real, \Real, 1_\Real)$, for which $L :=
\{1,\infty \}$ and $\tT:= \Real$, and whose operations are induced
by the standard operations $\max$ and $+$, cf.~
\cite{zur05TropicalAlgebra,IzhakianRowen2007SuperTropical}.
Other basic cases include $L = \Q_{\ge 0}$ and $L = \Q_{> 0}$.
Many more examples are given in
\cite{IzhakianKnebuschRowen2009Refined}.

We say that an element $c \in R$  is an $\ell$-ghost if $ s(c) $
is $\ell$-ghost. We also need the \textbf{$L$-surpassing}
relation:
$$a \lmodL b \dss{\text{ iff }} \left\{
\begin{array}{l}
 a=b  \quad  \text{ or }  \\[2mm]  a = b+c   \  \text{ for }  \  c  \  \text{ an }
 \  s(b)\text{-ghost}.
\end{array} \right.$$
\begin{defn}\label{surmor1}
  The \textbf{surpassing ($L,\nu )$-relation}
$\lmodWLnu$ is given by
 \begin{equation}  a \lmodWLnu b \qquad
\text{  iff } \qquad a \lmodWL  b \quad
\text{and}\quad
   a\nucong b .\end{equation} \end{defn}

   The following condition for surpassing plays an important role in this paper.
  We write $\2c$ for $c+c$.

   \begin{lem}\label{twosurp} $a +\2c  \lmodWL a$ for any $c\in R$. \end{lem}
\begin{proof} This is clear unless $c <_\nu a,$ but then $a+\2c = a$.
 \end{proof}

Suppose $\tT$ is a cancellative monoid, so that $R:=  \RR(L,\tT)$
is a uniform layered \domain0, where we identify $\tT$ with $R_1$.
We define the $\nu$-\textbf{topology} on $R$ to have a base of
open sets of the form
$$W  _{\a, \beta} = \{ a \in R: \a <_\nu  a < _\nu \beta  \}
\quad \text{ and } \quad W  _{\a, \beta; \tT} = \{ a \in \tT: \a
<_\nu a <_\nu \beta \}.$$ We call such sets \textbf{open
intervals}. For $\a, \beta$ tangible,  we write $[\a, \beta]$ for
the closure of $W _{\a, \beta}.$  We write $[\a, \beta]_\tT$ for
$\tT \cap [\a, \beta] := \{ a \in \tT: \a \le_\nu a \le _\nu \beta
\}$,   and call it a \textbf{closed tangible interval}.

\begin{defn}
A layered \semiring0 $R$ is \textbf{1-divisibly closed} if for
every $b \in R_1$ and $m \in \Net,$ there is $a \in R_1$ for which
$a^m = b.$
\end{defn}

For example, $D(\Q):= \Q \cup \Q^\nu$ is 1-divisibly closed.

%
%
%

\section{The function \semiring0 and a model for polynomials}\label{model}

Our approach to affine tropical geometry is to view varieties as
roots of polynomials, but sometimes we want variants of this
notion. In this section, we consider a model-theoretic framework.



\begin{rem}\label{clos}
As explained in~\cite{IzhakianRowen2007SuperTropical}, in contrast
to the situation for polynomials over algebras over an infinite
field, different polynomials over a \semiring0 may take on the
same values identically, viewed as functions. Thus, for any
\semiring0 $R$, and any set $S$, define  $\FunSR$ to be
 the set of   functions from~$S$ to ~$R$, made into
 a \semiring0\ in the usual way (via pointwise addition and
 multiplication).
   Our main interest in
 this paper is for $S = R^{(n)}.$
 Accordingly,  we
work in a given sub-semiring~$\tR$ of   $\Fun( R ^{(n)} , F),$
where $F$ is a suitable \semiring0 extension of $R$, such as  the
1-divisible closure of the \semiring0 of fractions of $R$, to be
explained below. Usually we take $R = F$ to be a 1-divisibly
closed \semifield0.
\end{rem}

\subsection{Polynomials and Laurent series}\label{pol1} Here are
the main settings for the theory. We denote the set of commuting
indeterminates $\{ \lm_1, \dots, \lm_n \}$ by $\Lm$, and write
$\Lm ^\bfi$, $\bfi = (i_1, \dots, i_n)$, for $\lm_1^{i_1} \cdots
\lm_n^{i_n}$. 

\begin{example}\label{polys}
  The following examples fit into this
context.

\begin{enumerate} \eroman

\item
 $R[\Lm] \subset \Fun( R ^{(n)} , F)$
denotes the 
polynomial \semiring0 over the \semiring0 $R$.
It is spanned over~$R$  by $\{\la_1^{i_1} \cdots \la
_n^{i_n}: i_1, \dots, i_n \in \Net \cup \{ 0\} \}$. 
\pSkip

\item The \textbf{Laurent polynomial \semiring0}
  $R[\Lm ,
\Lm ^{-1} ]$ denotes
the Laurent polynomial \semiring0 over~$R$.
It is spanned over~$R$ by $\{\la_1^{i_1} \cdots \la _n^{i_n}: i_1,
\dots, i_n \in
\Z \}$. 
\pSkip

 \item The \textbf{rational Laurent polynomial \semiring0} $R[\Lm ,
\Lm ^{-1} ]_{\rat}$ is defined analogously,
 spanned  over $R$ by the \textbf{rational
 monomials} $\{\la_1^{i_1} \cdots \la _n^{i_n}: i_1, \dots, i_n \in
\Q \}$.   By \cite[Remark~2.35]{IzhakianRowen2007SuperTropical},
if $W$ is a sub-\semiring0 of $\Fun( R ^{(n)} , F)$, then so is
$\sqrt W$. This yields an instant verification that the rational
Laurent polynomial \semiring0 is indeed a \semiring0, since it is
$ \sqrt{R[\Lm ,\Lm ^{-1}] }$.
\end{enumerate}
\end{example}

\begin{rem}\label{lexic} We utilize the \textbf{lexicographic order} on (rational)
monomials, where $\la_1 < \la_2 < \dots < \la_n.$ This enables us
to define the \textbf{leading monomial} of a rational polynomial
to be the one of highest lexicographic order.\end{rem}

\subsection{The model-theoretic approach}

These ideas may best be understood by means of model theory from
mathematical logic.

\begin{rem}\label{Modcomp} Throughout this paper, we let $\tL$ denote a
language whose elementary theory is model-complete. Our main
example is the language of ordered Abelian groups, since these
give rise to layered domains, as shown formally in
\cite[Proposition~3.11 and
Theorem~6.3]{IzhakianKnebuschRowen2010Categories}.\end{rem}

From now on, $F$ is a given $L$-layered 1-divisibly closed
1-\semifield0.

\begin{defn} $\Pol (S,F)$ denotes  the sub-\semiring0 of
$\FunSF$ comprised of functions defined in terms of the language
$\tL$.
  $\tR$ always denotes $\Pol(S,F)$ when $S$ is understood as given.
  \end{defn}

To avoid confusion,  $\Pol (F^{(n)},F)$ denotes  $F[\Lambda],$
taken in $n$ indeterminates, whereas $\Pol (F,F)$ denotes
$F[\lambda],$ taken in one indeterminate.  When dealing with
Laurent (or rational) polynomials, we explicitly use the notation
$F[\Lm , \Lm ^{-1} ]$ (or $F[\Lm , \Lm ^{-1} ]_{\rat}$).

\begin{rem}\label{Fun1} Any function $f\in\Pol (S,F)$ satisfies $f(\bfa) \nucong f(\bfa')$ whenever
$\bfa \nucong \bfa'$.
\end{rem}

  As
explained in
\cite[Corollary~5.27]{IzhakianKnebuschRowen2009Refined},
any theorem about roots of polynomials over arbitrary layered
\semifields0 can be verified by checking the 1-divisibly closed
layered \semifields0. This is formulated by Perri \cite{per11} as
a consequence of model-theoretic principles, and serves as a
useful tool for generalizing known facts about $\Real$ to
arbitrary divisibly closed \semifields0 and more varied
situations.

The model-theoretic approach enables us to unify the various
notions of  Example~\ref{polys}.
\begin{example}\label{model1} $ $
\begin{enumerate} \eroman
    \item  $\tL$  is the language of ordered Abelian
    groups, translated to the language of layered \domains0, but $\Pol(S,F)$ is defined without
    the  operation of taking inverses $(a \mapsto a^{-1})$. Then $\tR$ is
    the \semiring0 of polynomials.
 \pSkip

 \item $\tL$  is the language of ordered Abelian
    groups, translated to layered \domains0,including the operation of taking inverses. Then we have
    the \semiring0 of Laurent polynomials.
 \pSkip

    \item $\tL$  is as in (ii), together with the operation of taking $m$ roots
    $a \mapsto \root m \of a,$ for each $m \in \Net$. Then we have
    the \semiring0 of rational polynomials.
\end{enumerate}\end{example}

In a certain sense, polynomial and Laurent polynomial \semirings0
are local:

\begin{rem}[{\cite[Remark~4.5]{IzhakianRowen2007SuperTropical}}]
Suppose $L = L_{\ge 1}.$ Let $U$ be the group of invertible
elements of $\tR$. (In particular, $f(\bfa) \in \tT$ for each
$\bfa \in S$ and each $f \in U$.) For $F[\Lm ]$ (where $S =
F^{(n)})$, $U$ is just the set of multiplicative units of $\tT$,
and   for $F[\Lm ,\Lm ^{-1}]$, $U$ is   the set of tangible
monomials. In each case,
 $\tR \setminus U$ is the unique maximal ideal
of $\tR$. \end{rem}

\subsection{Decompositions of polynomials and their supports}

\begin{defn} Suppose  $f,g\in \FunSF $. We say that $f$
\textbf{dominates} (resp. \textbf{strictly dominates}) $g$ at
$\bfa \in S$ if  $f(\bfa ) \nuge g(\bfa)$ (resp. $f(\bfa ) \gnu
g(\bfa )$). We write $f \nuge g $ (resp. $f >_\nu g $) and say
that $f$ \textbf{dominates} (resp. \textbf{strictly dominates})
$g$ if $f(\bfa) \nuge g(\bfa)$ (resp. $f(\bfa) >_\nu g(\bfa)$) for
all $\bfa \in S.$
 We say that $f$ and $g$ are $\nu$-\textbf{equivalent}, written $f \nucong g$,  if $f \nuge g $
and $g \nuge f $.

Likewise, we write $f \lmodL g$ if $f(\bfa) \lmodL g(\bfa)$ for
all $\bfa \in S,$ and $f \lmodLnu g$ if $f(\bfa) \lmodLnu g(\bfa)$
for all $\bfa \in S.$
\end{defn}

Polynomials are best understood tropically as sums of monomials,
since their evaluations are the evaluations of the leading
monomials.

\begin{defn}\label{decomp}
Suppose $f = \sum   h_i \in \tR$  is written as a  sum of
monomials, and specify $h = h_j$ to be one of the $h_i$. Write
$f_h = \sum _{i \ne j}  h_i$ as a sum of monomials. The summand
$h$  is \textbf{inessential} in $f$ if $f = f_h$ as functions, and
$h$ is \textbf{essential} in $f$ if $f_h \not \nuge h$. We write
$\ef$ for the sum of the essential summands of~$f$.

%

The \textbf{support} $\supp(f)$ of   $f = \sum _{i } h_i$ is the
set of equivalence classes of the summands~$ h_i$; the number of
elements in the sum is called the \textbf{order} of the support,
written $|\supp(f)|$. The \textbf{tangible support} $\tsupp(f)$
 consists of equivalence classes of those  monomials $ h_i$ whose coefficients are tangible.

Two monomials are \textbf{support-equivalent} if they only differ
by their coefficient. A \textbf{decomposition} of~$f$ is a sum $f
= \sum h_i$ where each $h_i$ is not inessential and no pairs of
$h_i, h_j$ are support-equivalent.
\end{defn}

 Thus, a (rational) monomial has support of order 1, which is
tangible  iff  its coefficient is tangible. We discard all
inessential monomials of $f$ since they do not affect the value of
$f$ as a function.

\begin{defn}
A polynomial $f \in \tR$  is \textbf{tangibly spanned} if its support is all tangible, i.e., if all of its monomials have tangible coefficients. 

\end{defn}

\begin{rem} $ $

\begin{enumerate} \eroman

 \item The only polynomials taking on only tangible values on $R_1^{(n)}$ are the monomials
 with tangible coefficients.
 Thus, the tangibly spanned polynomials
 are precisely those polynomials with a decomposition as a sum of
 tangible monomials. \pSkip

\item
 Given support-equivalent monomials, one of them must dominate the
other
 (depending on which of $\al$ and~$\bt$ dominates in the
 definition).  Also, we can add any two monomials with the
same support, so we assume throughout that the monomials of a
decomposition of $f$ have disjoint support. Hence, the number of
monomials in a
 decomposition of $f$ is exactly the size of its support.
 \end{enumerate}
\end{rem}


The next result does not depend on the sorting set $L$.

\begin{lem}\label{dom1} Suppose a polynomial $f$
dominates $g$. Then, given any decompositions of $f$ and ~$g$, and
any $q \in \supp(f)\cap  \supp(g),$ the monomial $h'$ of $f$
having support $q$ must dominate the monomial $h''$ of $g$ having
support $q.$
\end{lem} \begin{proof} Otherwise $h''$ strictly dominates $h'$.
Take $\bfa$ such that $f(\bfa) \nucong  h'(\bfa)$, and note that
$$g(\bfa)\ge_\nu  h'' (\bfa) >_ \nu h' (\bfa) \nucong  f(\bfa),$$ a
contradiction.\end{proof}

\begin{example}\label{idealofzero1} $f = \la^2 + \la +3$ dominates
$g=2 \la,$ although the coefficient of $\la$ in $f$ is less than
the coefficient of $\la$ in~$g$; this does not contradict
Lemma~\ref{dom1} since $\ef = \la^2   +3,$ so $\supp(f)$ does not
include~$\la$. \end{example}

Since   addition never cancels in
 tropical mathematics, we have:
\begin{rem}\label{suppsum} $\supp(f+g)= \supp(f) \cup \supp(g)$.
\end{rem}

\begin{defn}\label{path} For $\bfa, \bfb \in F^{(n)},$ the \textbf{path} $\gm_{\bfa, \bfb}$ from $\bfa$
to $\bfb$ is the set $$\gm_{\bfa, \bfb}:= \{ \bfa^ {t  } \bfb^ {
1-t}  : t \in \Q, \ 0 \le t \le 1 \}.$$  A set $\tSS \subset
F^{(n)}$ is \textbf{convex} if for every $\bfa, \bfb \in \tSS$,
the path from
$\bfa$ to $ \bfb$ is contained in $\tSS$.  
\end{defn}

\begin{lem}\label{multmon}
By
\cite[Lemma~5.20]{IzhakianRowen2007SuperTropical}, one sees the
following, for any monomials $h_1$ and $h_2$ and
all  $\bfc \ne \bfa, \bfb$ in the path $\gm_{\bfa, \bfb}$ joining
$\bfa$ and $\bfb$:

 \begin{enumerate} \eroman
    \item If $h_1(\bfa ) \nuge h_2(\bfa )$ and
$h_1(\bfb)\gnu h_2(\bfb),$ then $h_1(\bfc )\gnu h_2(\bfc ) $;
 \pSkip

 \item If $h_1(\bfa ) \gnu h_2(\bfa )$ and
$h_1(\bfb)\nuge h_2(\bfb),$ then $h_1(\bfc )\gnu h_2(\bfc ) $;
 \pSkip

    \item If $h_1(\bfa ) \nuge h_2(\bfa )$ and
$h_1(\bfb)\nuge h_2(\bfb),$ then $h_1(\bfc )\nuge h_2(\bfc ) $.
\end{enumerate}
\end{lem}

 A (rational) polynomial $f$ is called
   a \textbf{binomial} if $|\supp(f)| = 2;$ i.e., $f$ has a decomposition as the sum of two (rational) monomials.
Given a decomposition of a polynomial $f = \sum _i h_i$ as a sum
of monomials~$h_i$, we define the set of \textbf{binomials of} $f$
to be the pairs of  monomials appearing in its decomposition.


\subsection{Layered components}

We consider some ideas that are standard in tropical mathematics
over $\Real$, but now can be put in a broader perspective.

\begin{definition}
 \label{Zar1} Suppose  $f = \sum_i h_i$  for monomials $h_i$.
 Define the $h_i$-\textbf{component} $D_{f,i}$
of $f$  to be
$$D_{f,i} := \{ \bfa  \in \tSS : f(\bfa) = h_i (\bfa)
\}.$$ For $k_1, \dots, k_n \in L$, the $(k_1, \dots,
k_n)$-\textbf{layer} of the component $D_{f,i}$ is $\{ (a_1,
\dots, a_n) \in D_{f,i} : s(a_j) = k_j$, $1\le j \le n.\}$
%



\end{definition}

 We need an extra assumption on
  $F$:

\begin{defn}\label{dense} A 1-\semifield0 $F$ is \textbf{dense}, if   every
path intersecting a component intersects the component at
infinitely many points.\end{defn}

%
%


\begin{lem}\label{mult11}
When $F$ is dense, any two monomials $h_1$ and
$h_2$ agreeing on an open set  $W$ of $F^{(n)}$ are equal.
\end{lem}
\begin{proof} We are given that $h_1(\bfa) = h_2(\bfa)$ for $\bfa \in W$. Suppose that $h_1(\bfb) \ne h_2(\bfb).$
The denseness hypothesis implies that the path $\gm_{\bfa,
\bfb}$ connecting $\bfa$ to $\bfb$ intersects $W$ nontrivially
(i.e., at a point $\bfc \in W $ other than $\bfa$), and one checks
easily using Lemma~\ref{multmon} that $h_1(\bfc) \ne h_2(\bfc).$
\end{proof}


\begin{thm}\label{uniquedec} Suppose a polynomial over a    1-divisibly closed \semifield0 $F$ has a
decomposition $f = \sum h_i$ into monomials. Then for any other
decomposition $f  = \sum h'_j$ the components with respect to
these two decompositions coincide, and the dominant monomials
coincide.
\end{thm}
\begin{proof} In view of Remark~\ref{Modcomp}, we may assume that
$F$ is dense, since $F$ can be enlarged into a dense 1-divisibly
closed \semifield0 $F$.

Take $\bfa \in D = D_{f,i}$ with respect to the first
decomposition, and suppose $\bfa \in D'_{f,j}$ with respect to the
second decomposition. Thus,  for any $\bfb \in D_{f,i}\cap
D'_{f,j},$ we have $$h_i(\bfb) = f(\bfb) = h_j'(\bfb),$$ implying
$h_i$ and $h_j$ coincide on the nonempty open set $D_{f,i}\cap
D'_{f,j},$ and thus are equal by Lemma \ref{mult11}. But then it
follows that the components coincide, and that $h_j' = h_i$ on
this component, implying $h_j' = h_i$.
\end{proof}

\begin{cor}\label{uniquede} Any decomposition of a polynomial $f$ as a  sum of  essential monomials
 is unique.
 Furthermore,  $f(\bfa)^m f(
\bfb)^{1-m} \ge f(\bfa ^m {\bfb} ^{1-m}) ,$ with equality holding
on a given path iff $f$ is a single monomial on
that path.
\end{cor}

\begin{rem}\label{motiv} Suppose $F$ is 1-divisibly closed. We would like to say that a
polynomial $f \in \tR$ cannot have two different dominating
tangible monomials on different points in the same  component.
Suppose this is false; i.e., $f(\bfa) = h_1(\bfa)$ and $f(\bfb) =
h_2(\bfb)$ for some $\bfa$ and $\bfb$ in a convex set. We would
have a contradiction if~$f$ takes a ghost value somewhere on the
path $\gm_{\bfa, \bfb}$ between $\bfa$ and $\bfb$, and we can find
this in principle by solving the equation
$$h_1(\bfa)^t h_1(\bfb)^{1-t} = h_2(\bfa)^t
h_2(\bfb)^{1-t},$$ or
$$\bigg(\frac{h_1(\bfa)h_2(\bfb)}{ h_2(\bfa)h_1(\bfb)}\bigg)^t =
\frac{h_2(\bfb)}{h_1(\bfb)},$$ which we could solve (for $t$) by
means of logarithms. This seems to entails an extra hypothesis
that $\tT$ is closed under taking logarithms, but in fact this
hypothesis can be removed, again by Remark~\ref{Modcomp}.\end{rem}

\section{Layered tropical geometry}

We continue to assume that $F$ is a layered 1-\semifield0, and
$\tR = \Pol(S,F).$ One of our main overall research objectives is
to connect tropical geometry to the algebraic structure of $\tR$.
The picture was painted in broad categorical strokes in
\cite{IzhakianKnebuschRowen2010Categories}, but here we only
consider the ideal structure. To get started, we need a
Zariski-type correspondence between algebraic varieties and ideals
of $\tR $. The following basic definition is taken from
\cite{IzhakianKnebuschRowen2009Refined}:

\begin{defn} 
The \textbf{layering map} of a function $f\in \FunSF$ is the map
$\vmap_f: \tSS\to L$ given by $$\vmap_f(\bfa) := s(f(\bfa)),
\qquad \bfa \in S.$$
\end{defn}

We write $\vmap_f \leq \vmap_g$ if $\vmap_f(\bfa) \leq
\vmap_g(\bfa)$ for  every $\bfa \in S.$

\subsection{Corner ideals and corner loci}\label{tropvar}

Customarily, given a polynomial, one takes its zero set. Here is
the analogous layered idea.

\begin{defn}\label{rootlev}  Suppose  $f
= \sum h_i \in \mcR$ is the decomposition of a polynomial. The
\textbf{corner support $\csupp_\bfa(f)$ of $f$ at} $\bfa$ is the
set
 $$ \csupp_\bfa(f) :=  \{ h_i \in \supp(f): f(\bfa) \nucong h_i(\bfa)\}.$$ We write $|\csupp_\bfa(f)|$
 for the order of $\csupp_\bfa(f)$.

The \textbf{corner locus} $\tZ_{\corn}(f)$ of a polynomial $f \in
\tR$ is $$\tZ_\corn(f) := \{ \bfa \in \tSS : |\csupp _\bfa (f)|\ge
2 \} .  $$ The \textbf{corner locus} $\tZ_{\corn}({I})$ of a
subset $ I \subset \tR$ is $ \bigcap_{f\in I}\tZ_\corn(f)$. Any
such corner locus will also be called an (affine) \textbf{corner
variety}.  The elements of the corner locus are called
\textbf{corner roots}.

\end{defn}


 For example, $\la _1^{5/3} + 7$ has the corner root
$4.2.$
\begin{rem}  The tangibly spanned  binomial $\la_1 ^{i_1}\cdots \la _n^{{i_n}}+ \al$
has the corner locus $$\bigg\{ (a_1, \dots, a_n) : \prod _{k=1}^n
a_k^{i_k }  \nucong \al\bigg\},$$\end{rem}
 The  corner locus defines much of the affine layered
 geometry, as described in \cite{IzhakianKnebuschRowen2009Refined}
and \cite{IzhakianKnebuschRowen2010Categories}.

%
%

\begin{lem}\label{csup2} There are three possibilities for $\csupp _\bfa
(f+g)$: Either  $\csupp _{\bfa} (f)$, $\csupp _{\bfa} (g)$, or a
set of monomials of $f+g$ whose values are $\nu$-equivalent  at
$\bfa$ to
the values of the monomials of $f$ corresponding to  $\csupp _{\bfa} (f)$. 
 \end{lem}
\begin{proof} There are three possible cases:

 \begin{itemize}
 \item $f(\bfa)>_\nu g(\bfa).$ Then $\csupp _\bfa (f+g)  =  \csupp _\bfa (f) ,$ so $|\csupp _\bfa (f+g)| \ge
 2$.
 \pSkip

\item $f(\bfa)<_\nu g(\bfa).$ Then $ \csupp _\bfa (f+g)  =  \csupp
_\bfa (g)$, so $|\csupp _\bfa (f+g)|\ge
 2$. \pSkip

 \item $f(\bfa)\nucong g(\bfa).$ Then $ (f+g)(\bfa )\nucong f(\bfa ) \nucong  g(\bfa ).$ 
\end{itemize}
  \end{proof}

 \begin{lem}\label{csup10}
$|\csupp _{\bfa} (f+g)|\ge
 2$ iff $|\csupp _\bfa (f^k+g^k)| \ge 2.$
\end{lem}
\begin{proof} This is clear unless $f(\bfa) \nucong g(\bfa),$ in which
 case $f^k(\bfa) \nucong g^k(\bfa),$ and thus we conclude with
 Lemma~\ref{csup2}.
  \end{proof}

   \begin{lem}\label{csup3} For any $k \in \Net,$ $\tZ_{\corn}(f+g)= \tZ_{\corn}((f+g)^k).$

 \end{lem}
 \begin{proof} Immediate from Lemma~\ref{csup10}.
  \end{proof}

\begin{lem}\label{csup11}
If $|\csupp _{\bfa} (f)|=|\csupp _{\bfa} (g)|=
 1,$ then $|\csupp _\bfa (fg)|=1$.
\end{lem}
\begin{proof} The hypothesis says that $f$ and $g$ both have a single dominant
monomial at $\bfa$, whose product is clearly the single dominant
monomial of $fg$ at $\bfa$.
  \end{proof}

We also quote a relevant result  from
\cite{IzhakianKnebuschRowen2009Refined}.

\begin{lem}[{\cite[Lemma~6.28]{IzhakianKnebuschRowen2009Refined}}]\label{csup1}
If $|\csupp _{\bfa} (f)|\ge
 2,$ then $|\csupp _\bfa (fg)|\ge
 2$ for all $g \in \mathcal R.$
\end{lem}

\begin{prop}\label{noroot}
 When $\tT := F_1$ is 1-divisibly closed, each polynomial $f$ with $|\supp(f)| \ge 2$ has a corner
 root, and any corner root of $f$ is a corner root of some binomial of $f$.
\end{prop}
\begin{proof} Suppose  $h_1 = \a \la_1^{i_1} \cdots \la _n^{i_n}$ dominates $f$ at $\bfa$ and
$h_2  = \beta \la_1^{j_1} \cdots \la _n^{j_n}$ dominates $f$ at
$\bfb $. Consider the path $\gm_{\bfa, \bfb}$. By
Lemma~\ref{multmon}, there can only be finitely many values of $t$
(notation as in Definition~\ref{path}) at which there is a change
of the dominant monomial of the path. Taking the smallest such
$t$, one can now easily solve $\a \la_1^{i_1} \cdots \la _n^{i_n}
= \beta \la_1^{j_1} \cdots \la _n^{j_n}$ to get the corner root,
since $\tT$ is divisibly closed. (Perhaps one has changed the
dominant monomial, but again, in view of Lemma~\ref{multmon}, this
process must terminate after a finite number of steps.)

The last assertion is obvious, by definition of corner root.
\end{proof}

Thus, binomials play a key role in the study of corner roots. 
On the other hand, we encounter some peculiar
corner loci.

\begin{example}\label{infgen0} $ $
\begin{enumerate}
\item $f_k =  \la_1^k + \la_2 + 0$ for $k \in \Net  .$ We consider
$\bfa = (a_1,a_2)$ with  $a_1, a_2 \in \tT.$
$$\vmap_{f_k}(\bfa) = \begin{cases}   3 \text{ for } a_1 = a_2 = 0; \\ 2 \text{ for } a_1 = 0 >
a_2 \quad \text{ or }\quad  a_2 = 0 >  a_1 \quad \text{ or }\quad
a_1 ^k =  a_2 > 0; \\ 1 \text{ otherwise. }
\end {cases}$$ \item ${\tI } =\{ f_k : k \in \Net \}.$ Now $$\vmap_{\tI}(\bfa) = \begin{cases}   3 \text{ for } a_1 = a_2 = 0;  \\ 2 \text{ for } a_1 = 0 >  a_2
\quad \text{ or }\quad  a_2 = 0 >  a_1 ; \\ 1 \text{ otherwise. }
\end {cases}$$
\end{enumerate}

 $\tZ_{\corn}({\tI }) = \{\bfa :  a_1 = 0 \ge  a_2
$ or $ a_2 = 0 \ge  a_1 \}. $
\end{example}

\begin{defn}\label{geomid2}  Given a subset $Z \subset S$,   define   $$ \tI_{\corn}(Z):= \{ f \in
\tR : |\csupp _\bfa (f)|\ge
 2, \ \forall \bfa \in Z\} .$$
 A  \textbf{corner ideal} is an
ideal of $\tR$ of the form $\tI _{\corn} (Z)$ for a suitable
subset
 $Z \subseteq \tSS$.
\end{defn}

\begin{rem}  $\tI_{\corn} (Z) =  \tI_{\corn} (\tZ
_{\corn}({ \tI _{\corn} (Z)})),$ so every \corner\ ideal arises
from a corner variety. Likewise, every corner variety arises
from a  \corner\ ideal.
\end{rem}

\begin{defn}\label{tropid01} A \textbf{$\nu$-closed ideal} of $\tR$ is a
\semiring0
 ideal $\tI $ satisfying the property that if
$f = \sum f_i  \in \tI $ and $g = \sum g_i$ are decompositions
into monomials with    $g_i \nucong f_i$ for each $i$, then $g \in \tI
$.

A \textbf{\lclo-closed ideal} of $\tR$ is a \semiring0
 ideal $\tI $ satisfying the weaker property that if
$f = \sum f_i  \in \tI $ and $g = \sum g_i$ are decompositions
into monomials with  $g_i \lmodLnu f_i$ for each $i$, then $g \in
\tI $.

\end{defn}


\begin{lem} Any corner ideal is $\nu$-closed. \end{lem}
\begin{proof} The corner locus
 only relies on the $\nu$-values of the monomials.\end{proof}

\begin{defn} Given any subset $A \subset \mcR,$ we define
$\root m \of  A$ to be the set $$ \root m \of  A := \{ f \in \mcR
\ds : f^m \in A\}.$$ \end{defn}
\begin{lem} If $A$ is
 a $\nu$-closed (resp.\lclo-closed) ideal of $\mcR$, then $\root m \of  A
$ is also a \lclo-closed ideal of $\mcR$.\end{lem}
\begin{proof} This follows at once from Lemma~\ref{csup10} and \cite[Remark~5.2]{IzhakianKnebuschRowen2009Refined}.\end{proof}

\begin{defn} The  \textbf{radical} of
$A$ is defined as $$\sqrt A := \bigcup _{m \in \Net} \root m \of
A.$$ The ideal $A$ is called \textbf{radical} if $A = \sqrt A$.
\end{defn} (In
particular $\Fun(S, F_{>1})$ is itself a radical ideal of the
 function semiring $\FunSF,$ and $\Pol(S, F_{>1})$ is a radical ideal of the
polynomial semiring $\PolSF.$)

The following motivational observation shows why radical (and in
particular prime) ideals are important.

\begin{rem} $\tI _{\corn} (Z)$ is a radical $\nu$-closed ideal of
$\mcR,$
 by Lemma~\ref{csup10}.
 \end{rem}


\subsection{Tropicalized ideals}

Our next objective is to identify ideals of special geometric
significance. We start with the specific ideals arising in the
transition from classical algebraic geometry to tropical geometry,
and then move on to intrinsic properties of ideals in the layered
structure.

Given an integral domain $K$ (in the classical sense) with a
valuation $v:K\to \tG,$  one takes the uniform layered
1-\semifield0 $F: = \RR(L,\tT)$ of Example~\ref{exs}, with  $\tT :
= \tG = F_1$ the tangible elements of $F$, and realize $v$ as   $v
: K \to \tT.$ (More generally, $K$ could be a valued monoid,
cf.~\cite[Definition~4.1]{IzhakianKnebuschRowen2010Categories}.)
Thus, the given operation on the ordered Abelian group $\tG$ is
taken to be multiplication in $F$, whereas addition in~$F$ is
induced from the given order on $\tG$. This takes us from the
classical world to the supertropical world,
 and is explained in categorical terms in
 \cite[Definition~5.6]{IzhakianKnebuschRowen2011CategoriesII}.
This is a supertropical valuation, as described in \cite{IKR1}.


 The map $v$ extends  to the polynomial map  $\tlv : K[\Lm ] \to \tT[\Lm ]$ given by $\la
_i \mapsto \la_i.$ In turn,  $\tlv$ induces a map
 $\{\text{ideals of }K[\Lm ]\} \to \{\text{ideals of }v(K) [\Lm ]\}$.

\begin{defn} For any subset $X \subset K[\Lm ],$
$\tlv(  X) := \{ \tlv(x) : x \in X\} $ is called the
\textbf{tropicalization} of~$X$. An ideal $ \tI$ of $F[\Lm]$ is
called \textbf{tropicalized} if $\tI \cap \tT[\Lm]$ is the
{tropicalization} of an ideal of $K[\Lm ]$.
\end{defn}

This concept involves some subtle difficulties.

\begin{defn}\label{gen0} A polynomial
$f \in \tR$ is \textbf{generated} by a subset $Y \subset \tR$ if
 $f = \sum_i f_ih_i$ for suitable $h_i \in
Y$ and  $f_i \in \tR$. We write $\langle S \rangle$ for the ideal
generated by a set $S$.

\end{defn}

\begin{example}\label{gen11}
 $\la+2$ is   generated
by the set $Y = \{ \la + 1, \la +3\}$,  as seen via the
calculation
$$ \la + 2 =  (\la +1) + (-1)(\la +3). $$
More generally, $ \la + \xl{2 }{\ell} =  (\la +1) + \xl{(-1)
}{\ell}(\la +3), $ for any $\ell \in L$.
\end{example}

\begin{example}~\label{badtrop} The polynomials $\la +1$ and $\la +2$
generate all of $K[\la]$ in the classical world,  whereas their
tropicalizations are the same when $v(1)=v(2)$, and thus generate
a proper ideal of the layered \domain0. Thus, the ideal $\langle
\tlv(X) \rangle$ generated by the tropicalization of a set $X
\subset K[\Lm ]$ need not be the tropicalization of the
ideal~$\langle X \rangle$ generated by $X$, i.e., $\tlv(\langle X
\rangle)$. (This difficulty is overcome by restricting one's
attention to Groebner bases.)\end{example}

In \cite[Theorem~2.1]{SpeyerSturmfels2004b} and the subsequent
discussion, the \textbf{tropical variety} of an ideal $A \subseteq
K[\Lm ]$ is defined as the intersection of tropical hypersurfaces
of all polynomials $\tlv(f)$ for $f$ in $A$. This is easily seen
to be the corner locus of the tropicalization of $A$,   so we
would like to obtain
 algebraic properties of tropicalized ideals.

\begin{rem}\label{suppsum1} Over any field $K$, if $h\in \supp(f) \cap \supp(g)$ for $f,g$ contained
in a $K$-subspace $V$ of~$K[\Lm ],$ then one can replace $g$ by
$\a g$ for suitable $\a \in K$ and assume that $f$ and $\a g$ have
the same monomial with support $h,$ so $h\notin \supp(f-\a g).$ In
other words, $f-\a g $ has support contained in $(\supp(f) \cup
\supp(g)) \setminus \{h\}.$
\end{rem}

Translated to layered \domains0, Remark~\ref{suppsum1} in
conjunction with Remark~\ref{suppsum} yields:

\begin{prop}\label{tropc} Suppose $f,g \in \tI $ where $\tI $ is a tropicalized ideal. Then for any $h\in \supp(f)\cap \supp(g)$
we can write $$f+g = p  + q,$$ where  $ h\in  \supp (p) \subseteq
\supp(f) \cap \supp(g),$  $q \in \tI ,\ $ and $ \supp(q )
\subseteq (\supp(f) \cup \supp(g)) \setminus (\supp (p )\cup \{ h
\}) .$
\end{prop}
\begin{proof} Adjust the respective pre-images $\bar f$ and
$\bar g$ (in $K[\Lambda]$) of $f$ and $g$ such that the monomials
with support $h$ cancel, and now write $\overline{q}$ for the sum
of the remaining monomials of $\bar f-\bar g $ that have common
support in both $\bar f$ and $\bar g$.   Then we write
$\overline{p}$ for the sum in $K[\Lm ]$ of the monomials of $\bar
f  $ not appearing in the support of $\overline{q}$. Thus, $\supp
\bar p$ is contained in $\supp g$ as well as $\supp f$, and
letting $p$ and $q$ be the respective tropicalizations of
$\overline{p}$ and $\overline{q}$, we have $h \notin \supp (q)$.
\end{proof}

In \S\ref{Eucl} we formalize the conclusion of
Proposition ~\ref{tropc} to restrict the class of ideals under
consideration.
%
%
%
%
%
We treated the layered Nullstellensatz briefly in
\cite{IzhakianKnebuschRowen2009Refined}, and need some relevant
observations here.

\begin{defn}\label{def:stackIdeal} A polynomial $f\in \tR$ is \textbf{covered} by $\tS \subseteq \tR$,
if for each \mcomp\ $D_{f,i}$ of $f$ there is some $g\in  \tS $
for which $\vmap_g \le \vmap_f$ on $D_{f,i}.$ A subset $\tS'
\subset \tS$ is \textbf{covered} by $\tS$ if each each $f\in \tS'$
is covered by $\tS $.\end{defn}

\begin{lem} If $\tS$ generates $\tI $, then $\tI$ is covered by $\tS.$\end{lem}
\begin{proof} Write $f = \sum p_i g_i$ for $p_i \in \tR$ and $g _i
\in \tI .$ Then for any monomial $h$ of $f$, we have some $p_i
g_i$ equal to $h$ on an open set, and thus on the $h$-component of
$f$, in view of
 Theorem~\ref{uniquedec}.
\end{proof}

In a sense, the layered Nullstellensatz of
\cite[Theorem~6.13]{IzhakianKnebuschRowen2009Refined} is the
converse, which we rephrase as follows:

\begin{thm}\label{Null3} \textbf{(Layered Nullstellensatz)}
Suppose $F$ is a 1-divisibly closed, archimedean, $L$-layered
1-\semifield0, $\tI  \triangleleft F[\Lm ]$ is \lclo-closed, and
$f \in F[\Lm ].$ Then $f$ is covered by~$\tI $
  iff $f \in \sqrt{\tI }.$
\end{thm}

This specializes to the following assertion for prime ideals:

\begin{thm}\label{Null3} \textbf{(Layered Nullstellensatz  for prime ideals)}
Suppose $F$ is a 1-divisibly closed, archimedean, 1-\semifield0,
$P \triangleleft F[\Lm]$ is a prime, \lclo-closed ideal, and $f
\in F[\Lm].$ Then $f$ is covered by~$P$
  iff $f \in P .$
\end{thm}

\subsection{Corner ideals}
We also can describe ideals of polynomials in terms of layering
maps.

\begin{defn}\label{geomid}  Given  $Z\subseteq S,$ define $\tR_Z = \Fun(Z,F) \cap \tR.$
Given any layering map $\vmap : Z\to L$, define
  $\tI _{\vmap}(Z)$ to be $$ \tI _{\vmap}(Z):= \text{$\{ f \in
\mcR_Z  : f(\bfa) $ is  $\vmap(\bfa)$-ghost, $\forall \bfa \in Z
\}.$}$$
A  \textbf{corner ideal} of $\mcR_Z $ is an ideal of the
form $\tI _{\vmap}(Z)$ for a suitable
 map $\vmap : Z\to L$. When $Z$ is understood, we write $\mcI_{\vmap}$ for $\tI _{\vmap}(Z)$.
\end{defn}

Strictly speaking, the notation for $Z$ is redundant, since we may
choose $\tSS$ as we please.  But often we start with $\tSS =
F^{(n)}$, and then take $Z$ to be a closed subset of $S$ with
respect to the layered component topology, so we have utilized the
symbol $Z$ for clarification. Recall that we assume $L = L_{\ge
1}.$

\begin{prop}\label{Zarcorresp}
   $\tI _{\vmap}:=\tI _{\vmap}(Z)\triangleleft \tR_Z$, and there   are 1:1 order-reversing correspondences between
the layering maps of ideals of $\mcR_Z $ and the corner ideals of
$\mcR_Z $, given by $\vmap \mapsto \tI _{\vmap}(Z)$ and $I \mapsto
\vmap_I$. \end{prop}
\begin{proof} Clearly  $\tI _{\vmap}(Z)$ is an ideal, since
 the  layering increases under multiplication.  For the second
assertion, one just follows the standard arguments in the Zariski
correspondence. Namely, we need to show that for any layering map
$\vmap$, defining the  geometric layered ideal $I = I_\vmap$, that
$\vmap _I = \vmap$ and $I_{\vmap_I} =  I$.

Clearly $I \supseteq I_{\vmap_I}$. But if $f\in I$ then by
definition $f \in \vmap_I.$ Hence $\vmap _I = \vmap$, so
$I_{\vmap_I} = I_\vmap = I$.
 \end{proof}

\section{Prime and maximal ideals of layered polynomial
\semirings0}\label{prime}

  We
are ready for our main algebraic interest in this paper, the
structure of the prime ideals of $\tR$, with special attention
paid to polynomial \semirings0 taken over a layered 1-\semifield0
$F$.

\begin{lem}\label{primeid} For any prime ideal $P$ of a uniform $L$-layered \domain0 $R$, either $e_k \in P$
for some $k\in L$ or $$ P = \bigcup _{\ell\in L}  e_\ell P_1,$$
where $P_1 = P \cap R_1$ is a prime monoid ideal of $R_1.$ Conversely, if
$P \cap R_1$ is a prime monoid ideal of $R_1,$ then $\bigcup _{\ell\in L}  e_\ell P_1,$ is a prime ideal of $R$.
\end{lem}
\begin{proof} Clearly $ P \supseteq \bigcup _\ell P_1
e_\ell.$ For the other direction, assume that $a\in P$, with $s(a)
= \ell.$ Then $a = e_\ell a_1$ for some  $a_1 \in R_1,$ so we are
done unless $a_1 \notin P,$ in which case $e_\ell \in P.$
\end{proof}

 Since the Nullstellensatz provides a correspondence from
geometric components to radical ideals, and every radical ideal is
the intersection of prime ideals, there must be many prime ideals
lurking around that are  not corner ideals. But we focus on
\corner\ ideals since there are too many \semiring0 ideals for
studying tropical geometry effectively.

\subsection{Prime ideals of supertropical polynomial semirings}

 Recall that the \textbf{standard supertropical} theory is obtained  for $L =  \{ 1,
 \infty\},$ where the transition map
$\nu_{\infty,1}$ is now the ghost map, which we denote as $\nu$.
Although this theory is a special case of the layered theory, it
has a different flavor, so we start with it  and then use the
layered theory
for refinement. 

\begin{rem}\label{compu} For any  positive $k,\ell  \in L$, we have
  $$(\xl{a_1 }{\ell}   + \xl{a_2 }{k} )(\xl{a_1 }{k}   +
\xl{a_2 }{\ell} ) =  \xl{(a_1^2 + a_2^2) }{ k\ell} + \xl{(a_1
a_2)}{\ell^2 +k^2},$$ which has layer $\ge   k\ell .$
It follows  for any $a$, that any  prime ideal $P$ of $\mcR$
containing $\mcR_{k\ell}$ also contains either $\xl{\la }{\ell} +
\xl{a}{k} $ or $\xl{\la}{k} + \xl{a}{\ell}$. In particular, in the
standard supertropical case, taking $k =\infty$ and $ \ell =1,$
$$(a_1 + a_2 ^\nu )(a_1 ^\nu + a_2) = (a_1^2 + a_1 a_2
+a_2^2)^{\nu}.$$\end{rem}

%
%
%

We quote the factorization in
\cite[Theorem~8.51]{IzhakianRowen2007SuperTropical}.

\begin{thm}
 \label{permprime} For any supertropical \semiring0, suppose $f = \sum_{i=1}^m f_i \in \Fun(\tSS,F)
$, for $m \ge 2$. Then
\begin{equation}\label{twofacts}
 \prod_{i < j } (f_i + f_j)  = g_1\cdots g_{m-1}\ ,
\end{equation}
where $g_1 = f= \sum_i f_i,$ $g_2 = \sum_{i < j } f_i f_j,$ $
\dots$, and $g_{m-1} = \sum_{i }\prod_{j \neq i} f_j$.
\end{thm}

The  role of binomials in prime ideals is found in the following
key observation.

\begin{cor}\label{binom1} In the standard supertropical theory, if $P$ is a   prime
ideal of $\tR$ and $f\in P$, then some binomial of $f$ belongs to~
$P$. \end{cor}

 Equality fails  in the layered version, since the layers in both
sides need not match, but we still have:

\begin{theorem}\label{permprime} For any $L$-layered \semiring0 $R$,
suppose $f = \sum_{i=1}^m f_i \in \Fun(S,R)
$, for $m \ge 2$. Then
\begin{equation}\label{twofacts2}  g_1\cdots g_{m-1}\, \lmodLnu  \prod_{i < j } (f_i + f_j),
\end{equation}%
where $g_1 = f= \sum_i f_i,$ $g_2 = \sum_{i < j } f_i f_j,$ $
\dots$, and $g_{m-1} = \sum_{i }\prod_{j \neq i} f_j$.
\end{theorem}

\begin{proof} Verifying Equation~\eqref{twofacts2} pointwise, let $a_i = f_i(\bfc)$ for $\bfc \in S$. It is enough to
check that
\begin{equation}\label{twofacts1}
b_1\cdots b_{m-1}\, \lmodL  \prod_{i < j } (a_i + a_j) \qquad
\text{and} \qquad b_1\cdots b_{m-1}\, \nucong  \prod_{i < j } (a_i
+ a_j)
\end{equation}
where $b_1 = \sum_i a_i,$ $b_2 = \sum_{i < j } a_i a_j,$ $\dots$,
$b_{m-1} = \sum_{i }\prod_{j \neq i} a_j$,  for $a_i \in R$.
 Let  $k_i = s(a_i)$.
 Rearrange the $a_i$ in descending $\nu$-order, i.e., with $$a_1
\nuge a_2 \nuge \cdots \nuge a_m.$$

First we assume that $a_i
>_\nu a_{i+1}$ for each $i$. Then $a_i + a_j = a_i$ for each $i<j$, whereas $b_i = a_1\cdots
a_i,$ so both sides of ~\eqref{twofacts1} are $a_1^m a_2^{m-1}
\cdots a_{m-1},$ and we actually get equality in this case. Thus,
we may assume that $a_i \nucong a_{i+1}$ for some $i<m$; we take
$i$ minimal such.  Then $s(a_i + a_{i+1}) = k_i + k_{i+1}$ whereas
for each $j>i,$
$$b_j =  a_1\cdots  a_{i-1} a_i +  a_1\cdots a_{i-1} a_{i+1} = a_1\cdots
a_{i-1}(a_i + a_{i+1})a_{i+2}\cdots a_j.$$ We conclude by
induction on $m$, replacing $a_{i}, a_{i+1}$ by $a_i + a_{i+1}.$
\end{proof}

\begin{example}\label{nonuniquefact2}
\begin{equation}\label{nonuniq} (\la_1+\la_2 +0)(\la_1\la_2 +\la_1+\la_2)\lmodLnu
(\la_1+0)(\la_2+0)(\la_1+\la_2),
\end{equation}
equality holding in the standard supertropical case.

In the standard supertropical case, the  principal ideal  $A =
\langle \la _1 + \la _2\rangle$ of $F[ \la_1, \la_2 ]$ is not
prime! Indeed, if $A$ were prime, Equation~\eqref{nonuniq} would
imply that $A$ contains $ \la_1+\la_2 +0$ or $\la_1\la_2
+\la_1+\la_2$, which is absurd, by an easy computation considering
degrees.

Likewise, in the standard supertropical case, the principal ideal
$A = \langle \la _1 + \la _2 + 0\rangle $ of $F [ \la_1, \la_2 ]$
is not prime, since otherwise $A$ would contain $ \la_1+0$, $
\la_2+0$, or $\la_1 +\la_2$, which again is seen to be impossible
by considering degrees.

In the more general layered case, equality fails in
\eqref{nonuniq}, but still
\begin{equation}\label{nonuniq1}
(\la_1+\la_2 +0)(\la_1\la_2 +\la_1+\la_2) \lmodLnu
(\la_1+0)(\la_2+0)(\la_1+\la_2).\end{equation}
 Thus, the layered closure of the principal ideal $A = \langle \la _1 +
\la _2\rangle$ of $F[ \la_1, \la_2 ]$ still is not prime.
\end{example}

Defining an ideal $P$ of $R$ to be $\nu$-\textbf{prime} if $ab \in
P$ implies $a$ or $b$ is $\nu$-equivalent to an element of~$P$, we
get the following immediate application of \eqref{twofacts2}:

\begin{cor}\label{binom2} Any $\nu$-prime ideal $P$ of $\tR$ contains a binomial.
In fact, any polynomial  $f\in P$ has a binomial in $P$.\end{cor}

For any 1-\semifield0 $F$, we call a polynomial $f \in \mcR$
\textbf{prime} if it satisfies the property that $f|gh$ implies
$f|g$ or $f|h$. (Thus, every prime polynomial is irreducible.
Conversely, unique factorization of all multiples of an
irreducible polynomial~ $f$ would imply that $f$ is prime.)

\begin{lem}\label{st1} $f\in \mcR $ is a prime polynomial
iff the ideal $\langle f \rangle$ is a  prime ideal of $\mcR
$.\end{lem}

\begin{proof} $f|g$ iff $g \in \langle f \rangle,$ so both
directions follow at once from the definition of prime ideal.
\end{proof}

\begin{lem} In  Lemma~\ref{st1}, if $F$ is a 1-\semifield0,
then any  prime ideal of the form $\langle f \rangle$ of $F[\Lm ]$
is a minimal prime ideal.\end{lem}

\begin{proof}  Suppose that $ P \subset \langle f
\rangle$ is a prime ideal. Then taking $g \in P$ of minimal
degree, clearly $g$ is irreducible, so we may assume that $g = f,$
and thus $\langle f \rangle = \langle g \rangle \subseteq  P
.$\end{proof}

On the other hand, Sheiner~\cite{Erez} has given an example of
non-unique factorization, which thus produces a non-prime
irreducible polynomial. Thus, the principal ideal of an
irreducible polynomial need not be prime.

\subsection{Examples of prime ideals of
$\tR$}

Various examples of prime ideals of
$\tR$ arise from geometry. Let
  $P_{\bfa;\corn}$ denote the set of polynomials whose corner loci contain a
given element $\bfa\in S$, i.e., $$ P_{\bfa;\corn} := \{ f \in \tR
: \bfa \in \tZ_\corn(f) \}.
$$

\begin{lem}\label{prime1}
$P_{\bfa;\corn}$ is a prime ideal, whose corner locus is precisely $\{ \bfa\}$.\end{lem}
 \begin{proof} If
$fg \in P_{\bfa;\corn}$, then $\csupp_\bfa(fg) \ge 2,$  implying
by Lemma~\ref{csup11} that $\csupp_\bfa(f) \ge 2$ or
$\csupp_\bfa(g) \ge 2,$ so $\bfa$ is a corner root of $f$ or $g$.
The last assertion is obvious since $ \la_1 +a_1, \dots, \la_n
+a_n  \in P_{\bfa;\corn}$, for $\bfa = (a_1,\dots, a_n).$
\end{proof}

Moreover, $P_{\bfa;\corn}$  is maximal among all  corner ideals,
since any larger corner ideal would have to be the corner ideal
whose corner locus is empty, and thus be all of $\tR.$ Note that
when $ \bfa  = ( a_1, \dots, a_n ) \in  F^{(n)},$ $P_{\bfa;\corn}$
contains $\langle \la_1 +a_1, \dots, \la_n +a_n \rangle
\triangleleft F[\Lm ].$ But   $P_{(3,3) ;\corn}$ also contains the
polynomial $\la_1 +\la_2 + 0$.

Analogously we can also get prime ideals by considering non-corner roots.

\begin{defn}\label{rootlev1}   An element $\bfa\in \tSS$ is a
\textbf{ghost root} of~$f$ if $s(f(\bfa))>1.$   The \textbf{$
\ell$-\textbf{locus} }  of a polynomial $f \in \mcR$ is
$$\tZ_\ell(f)  := \{ \bfa \in \tSS : s(f(\bfa))>\ell \},$$
The  $ 1$-locus of $f$ will also be called the  \textbf{ghost locus} of $f,$ since it
is the set of ghost roots.

 Let $\tZ_{\tng}(f) := \tZ_1(f) \cap F_1$ and let
 $$\text{$\tI  (Z):=\{ f \in \mcR: \lv(f(\bfa))>1$  for all
$\bfa \in Z\}.$}$$
\end{defn}

 Let
  $P_\bfa$ denote the set of polynomials whose ghost loci contain a
given element $\bfa\in S$.

\begin{lem}\label{prime11} $P_\bfa$ is a prime ideal.\end{lem}
 \begin{proof} If
$fg \in P_\bfa$, then $s((fg)(\bfa)) >1,$  implying   that
$s(f(\bfa))>1$ or $s(g(\bfa))>1.$
\end{proof}

\begin{example}\label{notfg}  Suppose more generally that $Z$ is \textbf{locally
irreducible at $\bfa$} (with respect to some given topology) in
the sense that there is no tangible neighborhood $W$ of $\bfa$ for
which $Z\cap W = (Z_1\cap W)\cup( Z_2 \cap W) $ for ghost loci
$Z_1$ and $Z_2.$ Then the set of polynomials whose ghost loci
contain $Z\cap W$ for a neighborhood~$W$ of~$\bfa$, is a prime
ideal, by the same argument.
\end{example}

\subsubsection{Polynomials in one
indeterminate over a supertropical \semifield0}
\label{primide}

We work in the supertropical setting, in which case a
1-\semifield0 $F$ is called a  \textbf{supertropical \semifield0},
and turn to the polynomial \semiring0 $\mcR$ in one indeterminate
over $F$.

\begin{lem}\label{tang11} Over any   1-\semifield0 $F$, any tangibly spanned polynomial     $f $ having
$\nu$-distinct tangible corner roots $ a_1, \dots, a_n$
 is divisible by $(\la+a_1) \cdots
(\la+a_n).$ \end{lem}
\begin{proof} We factor $f = (\la+b_1) \cdots
(\la+b_t),$  as a product of linear polynomials (with $b_j$
tangible) as in
\cite[Corollary~8.22]{IzhakianRowen2007SuperTropical}. Then, for
each $1 \le i \le n,$ $$s((a_i+b_1) \cdots (a_i+b_t))= s(f(a_i)) >
1,$$   since $a_i$ is a corner root of $f$, implying some $s( a_i
+b_j)>1,$  yielding $b_j \nucong a_i$ and thus $b_j = a_i$ since
both are in $F_1$. Reordering the $b_j$ such that $b_i = a_i,$ we
cancel $\la + a_i$ from $f$ and strike $a_i$ from the list, and
continue.
\end{proof}

The classification of all ideals is difficult even   in the standard
supertropical case. We start with some computations  based on the
list of irreducible polynomials given in
\cite[Example~8.6]{IzhakianRowen2007SuperTropical}.
%
%
Let us note a useful fact about roots.

\begin{lem}\label{dens} If,   for a  given $\a \in F,$  every tangible
 interval $W_{\a, \beta; \tT}$ ($\beta > _\nu \a$) contains a tangible 1-root of $f \in \mcR $, then
  $\tZ_{\tng}(f)$ contains
 a tangible interval $W_{\a, \beta; \tT}$ for some
 $\beta >_\nu \a.$ \end{lem}
 \begin{proof} Otherwise $W_{\a, \beta}$ contains a segment of a path which has
  tangible elements arbitrarily
 close to~$\a$ which are not 1-roots of $f$, which forces ~$f$ to
 have infinitely many tangible 1-roots close to $\a,$ which is impossible.\end{proof}


\begin{example}\label{exmp:prime} Some examples of prime ideals of
$\mcR $, where $a \in F$ is tangible.
\begin{enumerate}
\item    $P_a$ of  Lemma~\ref{prime11} is a prime ideal which, for
$\ell >1$, contains all multiples of $ \la +a $, $\xl{\la}\ell +a
_1,$ $ \la +\xl{a_2}\ell, $  and  $\la^2 + {\xl{a_2 }{\ell} } \la
+a_1a_2 $, whenever $a_1 \le_\nu a \le_\nu a_2.$ \pSkip

  \item Let $P_a^{\lft}$ be the set of polynomials whose  ghost loci contain a
  closed tangible interval  starting with~$a$, i.e., of the form
  $\{ b\in \tT: a \le_\nu b <_\nu a_1\}$ for some $a_1$ with $a  <_\nu
  a_1.$
$P_a^{\lft}$ is a prime ideal, since if $fg \in P_a^{\lft}$ then,
by Lemma~\ref{dens}, some closed tangible interval starting with
$a$ is in the ghost locus of $f$ or $g$, say of $f$. For $\ell >1$, $P_a^{\lft}$
contains all multiples of  $ \la +\xl{a_1}\ell  $
and
$\la^2 + \xl{a_1 }{\ell}\la +aa_1 $ whenever $  a <_\nu a_1 .$
\pSkip

 \item Let $P_a^{\rgt}$ be the set of polynomials whose  ghost loci contain a
  closed tangible interval  terminating with $a$, i.e., of the form
  $\{ b\in \tT: a_1  <_\nu  b \le_\nu  a \}$ for some $a_1$  with $a_1  <_\nu
  a.$ $P_a^{\rgt}$ is a prime
  ideal, for the same reason as in (2).   For $\ell >1$, $P_a^{\rgt}$ contains all multiples of  $
 \xl{ \la}\ell +  a_1 $  and $\la^2 + \xl{a }{\ell}\la
+aa_1 $  whenever $ a_1 <_\nu a.$

\end{enumerate}
\end{example}

%
%

 Notation as in   Example
 \ref{exmp:prime}, $P_a^{\lft}, P_a^{\rgt}  \subset P_a,$ and $P_{a;\corn}  \subset P_a$.

\begin{lem} Assume that the archimedean   1-\semifield0 $F$ is complete with respect to the
$\nu$-topology. In the standard supertropical case, if a prime ideal $P\triangleleft \mcR $ contains
$\tI _1(Z),$ where $Z$ is a closed tangible interval (in the
$\nu$-topology) which is not a point, then $P$ contains $P_a^{\lft}$ or $ P_a^{\rgt}$
for some $a\in Z$.
\end{lem}
\begin{proof} Take $Z_0 = Z$. Inductively, given $Z_i$,
 write $Z_i = \bigcup_{j=1}^k Z_{i,j}$ for closed tangible intervals $Z_{i,j}$, $1 \le i \le k,$
 say of length each at most half of that of $Z_i$.
Take quadratic polynomials $f_{i,j}$ having tangible corner
locus $Z_j$. Then $f_{i,1} \cdots f_{i,k} \in P,$ implying some
$f_{i,j} \in P$. Now let $Z_{i+1} = Z_{i,j}$ and continue the
procedure. We thus divide $Z$ into smaller and smaller tangible
intervals, which converge to some $a$, and~ $P$ contains the
corresponding quadratic polynomials.  But also $\la + a^\nu\in P$
or $\la^\nu +a \in P$, by Remark~\ref{compu}. Hence, $P_a^{\lft}$
or $ P_a^{\rgt}$ is contained in $P$.
\end{proof}

Since the prime ideals  $P_a^{\lft}$ and $ P_a^{\rgt}$ of the
lemma are  not corner ideals, we conclude:

 \begin{cor}\label{corpr}
The only prime corner ideals of $\mcR $ are the $P_a$.
\end{cor}
(This result also is a direct consequence of Theorem~\ref{Euclid}
below.)

\begin{prop}\label{class1}  If $a<_\nu  b$ are tangible, then the ideal $P$ generated by
$P_{a;\corn}$, $P_{b;\corn}$,  and $P_c$ for all $c\in \tT$
satisfying $a < _\nu c< _\nu b$ is prime. Conversely, any
\lclo-closed prime ideal $P$  of $\mcR $ containing $P_{a;\corn}$
as well as  $P_{b;\corn}$ also contains~$P_c$ for all $c$
satisfying $a < _\nu c< _\nu b$.
\end{prop}
\begin{proof} Suppose $fg \in P$. We need to show that $f \in P$ or $g\in P$.
Then    $\tZ_{\tng} (fg)$ contains some point~$c$ with $a \le_\nu
c\le_\nu b$, since otherwise the interval $W _{a, b}$ is in the
complement set of  $\tZ_{\tng} (fg)$ but is not contained in the
complement set of any element of $P$, contrary to the
Nullstellensatz. Hence, $c$ is a root say of $f$, implying $f \in
P$ unless $c=a$ or $c=b.$ We may assume that $c=a$, and are done
unless~$a$ is a corner root of $f$; i.e., $\la +a$ divides $f$. If
$\la +a \in P$ then $f\in P_a^{\lft}$ or $f\in  P_a^{\rgt}$ and we
are done; otherwise, we write $f = (\la +a)h$ and have $hg \in P$;
applying induction on the degree yields $h \in P$ or $g \in P$,
proving~$P$ is prime.

The converse is   an easy application of the Nullstellensatz,
since the complement set of $\tZ_{\tng}(\la+c)$  has two
components, one contained in the tangible complement of the corner
locus of  ~$\la +a$ and
the other contained in the tangible complement of the corner
locus of $\la + b .$\end{proof}

\subsection{Prime corner ideals}

There are so many \semiring0 ideals that we want to cut them down
in some way which does not affect the applications to tropical
geometry. Thus, we turn to a more intensive study of corner
ideals.

\begin{thm}\label{exchange15} Suppose for polynomials $f_1,\ f_2,\ h_1,$ and
$h_2$ that
 $f_1+h_1$  and $f_2+h_2$ are  in $A: =   \tI _{\corn}(Z)
\triangleleft \mathcal R$, with $\supp(f_1)\cap \supp(f_2) =
\emptyset$, and furthermore that  $h_1(\bfa) \nucong  h_2(\bfa)$
for each $\bfa\in  Z .$ Then $h_1h_2(f_1+f_2) \in A. $\end{thm}
\begin{proof} We need to show  that $|\csupp_\bfa(h_1 h_2 (f_1+f_2))|\ge 2$
for
 $\bfa \in Z$. We are done unless $|\csupp_\bfa(h_j)| < 2$ for $j = 1,2$.
But by hypothesis,  $h_1(\bfa)\nucong h_2(\bfa)$. We consider each
possible situation.
\begin{enumerate}

    \item[--]
 If $f_1(\bfa)
> _\nu
 h_1(\bfa)$ and  $f_2(\bfa) > _\nu
 h_2(\bfa)$,  then $|\csupp_\bfa( f_1) | = |\csupp_\bfa(f_1+h_1)|\ge 2,$  and likewise
 $|\csupp_\bfa( f_2) |= |\csupp_\bfa( f_2+h_2)|\ge 2,$
so Lemma~\ref{csup2} implies $|\csupp_\bfa(f_1+f_2)|\ge 2.$ \pSkip

\item[--] If   $f_1(\bfa)  \nucong
 h_1(\bfa)$ and  $f_2(\bfa) > _\nu
 h_2(\bfa)$,  then $$|\csupp_\bfa(f_1+f_2)| = |\csupp_\bfa( f_2) | = |\csupp_\bfa(f_2+h_2)|\ge
 2.$$ \pSkip

\item[--] If   $f_1(\bfa)  \nucong
 h_1(\bfa)$ and  $f_2(\bfa)  \nucong
 h_2(\bfa)$,  then $|\csupp_\bfa(f_1+f_2)| \ge 2,$ since we get one
 dominant monomial at $\bfa$ from each $f_i$. \pSkip

\item[--]  Finally, if  $f_1(\bfa) <_ \nu
 h_1(\bfa)$, then   $|\csupp_\bfa( h_1) |= |\csupp_\bfa(f_1+h_1)|\ge 2.$
\end{enumerate} In all cases we conclude by means of
 Lemma~\ref{csup1}.
 \end{proof}

\begin{cor}\label{exchange14} Suppose for polynomials $f_1$, $f_2,$ and
$h$, that
 $f_1+h$ and $f_2+h$  are  in $\tI _{\corn}(Z)
\triangleleft \tR$. Then $h (f_1+f_2) \in \tI _{\corn}(Z)
.$\end{cor}
\begin{proof} Take $h_1 = h_2 = h.$ Then $h ^2(f_1+f_2) \in \tI
_{\corn}(Z),$ implying $h (f_1+f_2) \in \tI _{\corn}(Z).$
 \end{proof}

\begin{cor}\label{exchange2} Suppose $A$ of Theorem~\ref{exchange15}
is a prime corner ideal, with $h_1,h_2 \notin A.$ Then $f_1+f_2 \in
A.$\end{cor}

%







\subsection{Ideals of $\ell$-loci}

 Here is an alternate approach, perhaps more in
line with \cite{IzhakianRowen2007SuperTropical}, where one would
take $\ell = 1.$


\begin{thm}\label{exchange} Suppose for polynomials $f_1,\ f_2,\ h_1,$ and
$h_2$ that
 $f_1+h_1,$ and $f_2+h_2$ are  in $A: = \tI   _\ell(Z)
\triangleleft \mathcal R$, with    $h_1 \nucong h_2$. Then $(f_1+f_2)
h_1h_2 \in A  .$\end{thm}
\begin{proof} We need to show for any
 $\bfa \in Z$ that $h_1(\bfa ),$ $h_2(\bfa ),$ or $f_1(\bfa )+f_2(\bfa )$ are $\ell$-ghost. So assume that
  $h_1(\bfa)$ and $h_2(\bfa)$ are not $\ell$-ghost. Then  $f_1(\bfa) \ge _\nu
 h_1(\bfa)$ since $f_1(\bfa) +
 h_1(\bfa) $ is $\ell$-ghost, and likewise $f_2(\bfa) \ge _\nu
 h_2(\bfa)$. We assume that  $f_1(\bfa) \ge _\nu
 f_2(\bfa)$.

If $f_1(\bfa)  >_\nu  h_1(\bfa),$ then
 $f_1(\bfa) $ is $\ell$-ghost, in which case $(f_1+f_2)(\bfa) =f_1(\bfa) $ is $\ell$-ghost and we are done.
 Thus, we may assume that  $f_1(\bfa)  \nucong
 h_1(\bfa)$. But now $f_1(\bfa) \nucong f_2(\bfa) \nucong h_1(\bfa)$, so $\bfa$ is an $\ell$-root of $f_1+f_2$.
 \end{proof}

\begin{cor}\label{exchange1} Suppose for polynomials $f_1$, $f_2,$ and
$h$, that
 $f_1+h$ and $f_2+h$  are  in $A: =  \tI   _\ell(Z)
\triangleleft \tR$. Then $h (f_1+f_2) \in \tI _\ell (Z)
.$\end{cor}
\begin{proof} Take $h_1 = h_2 = h,$ noting that $h_1+h_2  \in A.$
 \end{proof}

\begin{cor}\label{exchange12} Suppose $A$ of Theorem~\ref{exchange}
is a prime ideal, with $h_1, h_2 \notin A.$ Then $f_1+f_2 \in
A.$\end{cor}

This leads to an intriguing notion. Given polynomials $f$ and $g$,
we say that a monomial $h$ is $\nu$-\textbf{common} to $f$ and $g$
if $f$ and $g$ both have essential monomials $\nu$-equivalent to
$h$.

\begin{cor}\label{exchange13} Suppose $P$
is a prime  ideal. If $g_1,g_2 \in P$, then either the sum
of all monomials $\nu$-common to $g_1$ and $g_2$ is in $P$, or
$g_1 + g_2 \in P.$\end{cor}
\begin{proof} Write $g_1 = f_1 + h_1$ and $g_2 = f_2 +h_2$, where
$h_i$ is $\nu$-equivalent to the part that is $\nu$-common
with~$g_1$ and $g_2.$ Then  we can apply Theorem~\ref{exchange}.
  \end{proof}

\subsection{Exchange ideals}

 As in classical algebra, the theory of ideals of polynomial semirings in
several indeterminates is much more difficult than in one
indeterminate. In the tropical setting, the situation is even
worse in some regards, as exemplified in
\cite[Example~8.52]{IzhakianRowen2007SuperTropical}. Nevertheless,
 we are interested in studying   ideals and their impact on geometry, in particular in generating corner ideals
 by means of binomials
(insofar as we can). Accordingly, we refine the definition of
ideal in order to focus on tangible corner roots of polynomials.

\begin{defn}\label{def:tropIdeal}
An \textbf{\aIdeal} of $\mathcal R$ is a \lclo-closed ideal~$A$
which satisfies the property:
\begin{enumerate}
     \item[  ] (``Exchange law'') If $A$ contains $f+h$ and
 $g+h $ with $\supp(f) \cap \supp(g) = \emptyset,$
  then  either $h \in A$ or    $f+g \in A.$
\end{enumerate}
The ideal $A$ of $\mathcal R$ is an \textbf{\bIdeal} (``m'' for
``monomial'') if it satisfies the weaker condition:
\begin{enumerate}
    \item[  ] (``m-Exchange law'') If  $A$ contains $f+h$
    and $g+h $ with $\supp(f) \cap \supp(g) = \emptyset$, where $h$ is a tangible
 monomial not in  $\supp( f),$
  then  also    $f+g \in A.$ \end{enumerate}
A \textbf{prime \bIdeal}\ is a prime ideal that is also an
\bIdeal.
\end{defn}


Our motivating example of an \aIdeal\ is the prime corner ideal,
which is an \aIdeal \ by~Corollary~\ref{exchange2}.
 We use the $m$-exchange law mostly in the special case that $g$ is
a constant $\a;$ it basically says   that we can replace a
monomial $h$ by the constant $\a$ in any polynomial of $A$.

 \begin{example}\label{tropid}
   Suppose a proper \bIdeal\ $A\triangleleft \mcR $ contains two  tangibly spanned binomials $f_1 =
h_1 + \a h_2$ and $f_2 = h_1 + \beta h_2$, with $\beta \gnu \a$
and $h_1$ tangible; then $\beta h_2 = (\a + \beta)h_2 \in A.$ If
$\beta$ were tangible, then $A$ would be improper. Thus, $\beta\in
F_{>1}$, and $f_2 = f_1 + \beta h_2\lmodL f_1 .$ \pSkip
\end{example}

\begin{lem} If $A$ is a \lclo-closed \aIdeal \ of $\mcR$,
then $\root m  \of A$ is also an \aIdeal, for any~$m$.\end{lem}
\begin{proof} For $f+h,\, h+h' \in \root m  \of A,$  we have  (in
view of Lemma~\ref{csup3}), $h^m + (h')^m  \in A$ and $f^m +h^m
\in A,$ implying $f^m+(h')^m\in A,$ and thus $f+h' \in \root m \of
A.$
\end{proof}

\subsection{Binomials in ideals}

To understand how binomials generate ideals, we need to see which
binomials in an ideal are consequences of the others. It is
convenient to work in $F [ \Lm , \Lm ^{-1}]$.

The following result is really about localization.

\begin{prop} The natural
  injection $\phi: F[\Lm ] \to F[\Lm ,\Lm ^{-1}]$ induces a lattice injection from
     \{\aIdeal s of~$F[\Lm ] $ not containing monomials\} to    \{\aIdeal s of $F[\Lm ,\Lm ^{-1}]$\},
     which is 1:1 on the prime \aIdeal s.
\end{prop}
\begin{proof} Write $\htR = F[\Lm ,\Lm ^{-1}]$; our lattice injection is to be given by $A \mapsto \htR A.$
First we check that if $A$ is an \aIdeal \, of $ F[\Lm ]$, then
$\htR A$ is an \aIdeal \,  of $\htR$.   In view of the common
denominator property, any binomial of $ \htR$ has the form $\frac
hs +\frac {h'}s$, for some monomial $s$. We need to check that if
$\htR A$ contains a binomial $\frac hs +\frac {h'}s$ and
   if $ f+\frac hs  \in \htR A$,
  then $\htR A$ also contains  $f+\frac {h'}s.$ But this is
  clear:  $A$ contains $ sf+h$, so
  $A$ contains $ sf+h' $ by the exchange law in $F[\Lm ]$, and
  we just divide by $s$.

Finally, if $\htR A_1 = \htR A_2,$ for prime \aIdeal s $A_1,A_2$
of $F[\Lm ]$, then we claim that $A_1 = A_2$. It is enough to show
that each element $f \in A_1$ belongs to $A_2.$ But $\frac{ f
}{\fone} \in \htR A_1 = \htR A_2,$ so $sf  \in A_2$ for some
monomial $s,$ implying $f \in A_2$.
\end{proof}

Localizing  makes the bookkeeping easier.
Let us look closer at how the exchange law acts on \eBin s in a
given  \bIdeal \,  $A$ of $F [ \Lm , \Lm ^{-1}]$. Cancelling out
suitable powers of the $\la _i^{\pm 1},$ we write these in the
form $\la_1 ^{i_1}\cdots \la _n^{{i_n}}+ \a,$ where $(i_1, \dots,
i_n)\in \Z ^{(n)},$ which we order under the lexicographic order
$<_{lex}$. \pSkip

\begin{rem}\label{binomgen}
 If $\la_1 ^{i_1}\cdots \la _n^{{i_n}}+ \a \in A,$ for $\a$
tangible, then dividing through by $ \a \la_1 ^{i_1}\cdots \la
_n^{{i_n}}$ yields $$\la_1 ^{-i_1}\cdots \la _n^{{-i_n}}+
\a^{-1}\in A.$$ and thus \begin{equation}\label{Laur} \a\la_1
^{-i_1}\cdots \la _n^{{-i_n}}+ \fone \in A.\end{equation} \pSkip
\end{rem}

\begin{lem} If $A$, an m-exchange ideal,
contains two given binomials $$h_\bfi = \la_1 ^{i_1}\cdots \la
_n^{{i_n}}+ \al, \qquad h_\bfj = \la_1 ^{j_1}\cdots \la
_n^{{j_n}}+ \bt,$$
 with $\al$ tangible and $(i_1, \dots, i_n) \ne (j_1, \dots, j_n)$, then $A$ contains the  binomial $\la_1
^{j_1-i_1}\cdots \la _n^{{j_n-i_n}}+\gm$.\end{lem}
\begin{proof} Applying the m-exchange law to Equation~\eqref{Laur},   $A$ also contains the binomial
$$\al \la_1
^{-i_1}\cdots \la _n^{{-i_n}} +  \beta \la_1 ^{-j_1}\cdots \la
_n^{{-j_n}} =  \al \la_1 ^{-j_1}\cdots \la _n^{{-j_n}}(  \la_1
^{j_1-i_1}\cdots \la _n^{{j_n-i_n}}+\gm),$$ where $\gm = \frac
{\beta}{\al}\in F.$\end{proof}

\begin{thm}\label{genbin} Any
set of tangibly spanned binomials is generated by at most $n$
tangibly spanned binomials. \end{thm}
\begin{proof} Label any  binomial  $\la_1 ^{i_1}\cdots \la
_n^{{i_n}}+ \gm$ by the vector $(i_1, \dots, i_n)\in \Z^{(n)}$
(disregarding $\gm$), and write~$G_A$ for the set of such  vectors
corresponding to \eBin s of the  \aIdeal \,  $A$.  By  Remark
\ref{binomgen}, $G_A$ is closed under subtraction, and thus is a
group. It follows that any set of rows of $t$~vectors in $G_A$ can
be transformed by the standard procedure of Gauss--Jordan
elimination into rows in which each of the first $t$ columns has
at most one nonzero entry. Translating back to \eBin s, we see
that  in any
 proper $m$-\aIdeal\ $A$ of $F [ \Lm , \Lm ^{-1} ]$, any
set of tangibly spanned binomials is generated by at most $n$
tangibly spanned binomials.    \end{proof}

%
%
%
%
%

%

\section{Generation of layered ideals}\label{primeideals}

We briefly considered generation of ideals  in
Definition~\ref{gen0}.


%
%
%


\subsection{Generation by irredundant binomials}

\begin{defn} A set of polynomials $\tB = \{ f_1, \dots, f_m \}$ is
\textbf{redundant} if $\tB$  belongs to the \aIdeal\ generated by
$\tB \setminus \{ f_j \},$ for some $j$. Otherwise $\tB$ is called
\textbf{irredundant}.
\end{defn}

\begin{example}\label{gen30} If $\tB =\{ h+
a,\ h+ b\}$  with $h+ a$  tangibly spanned, then either the set
$\tB$ is redundant, or $\tB$ generates all of $F$ (as an \bIdeal).

Let $A$ be the exchange ideal generated (as an exchange ideal) by
$\tB$. We demonstrate the assertion by subdividing it into four
cases:
\begin{enumerate} \eroman   \item $a  \nucong
 b$. Then $\tB$ is redundant. \pSkip
    \item
$s(b) >1 $ and $a <_\nu b$.
Then $h+ b = (h+a) + b \lmodL h+a,$ implying $h+a$  generates
$h+ b$.
\pSkip \item $s(b) >1 $ and $a >_\nu b$. By the m-exchange law,
$A$ contains $a+b=a,$ and thus contains $\fone,$ so is improper.
\pSkip \item   $b$ is tangible, with $a \not \nucong
 b$. Then by the m-exchange law, $A$  contains $a+b \in
\{a,b\}$, and thus contains $\fone,$ so is improper.
\end{enumerate}
\end{example}

\begin{thm}\label{ngener0} Suppose $F$ is 1-divisibly closed. For
every \bIdeal \,  $A$ of $ F[\Lm , \Lm ^{-1}]$ (for $\Lm  = \{
\la_1, \dots, \la _n \}$), its set of tangibly spanned binomials
is generated (via the exchange property) by a set of at most $n$
irredundant binomials of $A$.
\end{thm}
\begin{proof}  Take some  polynomial $f = \sum \a_{\bold i} \La ^{\bold i} + g$
in $A$, where all the $\a_{\bold i}$ are tangible, and $g\in
\tG[\Lm ,\Lm ^{-1}]$. Let $f_{{\bold i}, {\bold j}} =\La ^{{\bold
i}-{\bold j}} + \frac{ \a_{\bold j}}{\a_{\bold i}},$ taken over
all (finitely many) ${{\bold i}, {\bold j}}$ such that $\a_{\bold
i}, \a_{\bold j}\ne \fzero.$ Then Example \ref{gen30} likewise
shows that for any  fixed tangible monomial $h$, any finite set
$\{ h + \al\}$ of binomials (where
 each $\al \in F$ is tangible) is generated
by a single one of them, so we conclude with Theorem~\ref{genbin}.
\end{proof}

\subsection{Factorization of binomials}

To decompose binomials further, we say that a polynomial $f $ is
$L$-\textbf{reducible} if there are polynomials $g,h$ of degree
$\ge 1$ such that $gh \lmodWLnu f$.

\begin{lem} When $F$ is 1-divisibly closed and $\a \in F_1$, the binomial $\la_1 ^{i_1}\cdots \la _n^{{i_n}}+ \al$
is $L$-irreducible iff the integers $i_1, \dots, i_n$ are relatively
prime.
\end{lem}
\begin{proof}$(\Rightarrow)$ By the contrapositive. Assume $d$ divides each $i_1, \dots, i_n.$ Then
$$ \bigg(\la_1 ^{\frac{i_1}d}\cdots \la _n^{\frac{i_n}d}+ \root d \of \al\bigg)^d\lmodWLnu\la_1 ^{i_1}\cdots \la _n^{{i_n}}+ \al . $$

$(\Leftarrow)$ The product of polynomials can be a binomial iff
all of the intermediate terms are inessential, which cannot happen
when the exponents are relatively prime. \end{proof}

\begin{prop}\label{factorbin} If $F = \clF$ and $f\in F \pl \Lm  \pr$ is a
 \eBin, then $f$ can be $L$-factored as a product of a
monomial $h$ times a power $g^m$ of an $L$-irreducible binomial,
in the sense that $hg^m \lmodWLnu f.$ \end{prop}

\begin{proof} Let us write $f = \a \Lm^\bfi + \beta \Lm^\bfj$.
Factoring out $\beta,$ we may assume that $\beta = \fone.$ It is
convenient to work in $F \pl \Lm , \Lm ^{-1} \pr$, since then we
may divide by $\Lm^\bfj$ and assume that $f$ has the form $\a
\Lm^\bfi +\fone$. We are done unless the full closure of $f$ has
some monomial on the line connecting $\bfi$ to $(0,\dots, 0)$. In
other words,   $\bfi = m\bfk$ for suitable $m\in \Net $ and $\bfk$ . But then $
{\a} \La^{\bfi}$,  a monomial of $f$,   is the $m$-th power
of $  \root m \of {\a}\La^{\bfk}$. We conclude using the
lemma.\end{proof}

\subsection{Layered generation of polynomials}

We conclude this section with an explicit discussion of generation
of polynomials in one indeterminate, relying heavily on \cite{IzhakianRowen2007SuperTropical}.

 \begin{example}\label{linear2} $ $ Suppose $F$ is a  layered 1-\semifield0, and $a,b \in
 F$
 with $b \lmodL a$.  \pSkip
\begin{enumerate} \eroman
    \item
 The $\lmodL$-closed ideal generated by $\la + a$  contains $\la +
b $; this is clear if $b \nucong a,$ so we assume that $b \ne a,$
in which case $\la + b  = (\la + a) + b  $.\pSkip
\item
 Any ideal containing $f_1 =  \la + a$ and $f_2  = \xl {\la}\ell +
 b$ ($\ell$ arbitrary) also contains $\lm + c$ for all $c \in \tT $ with $a
 < _\nu c   <_\nu  b ,$ since
 $$\la +c = (\la  + a) + \frac {c}{b}\bigg(\xl {\la}\ell +
 b\bigg).$$

  \item  If $a_1  <_\nu a_2  \le _\nu  b,$
 then the polynomial $\la ^2 + \xl {b}\ell \la + a  _1b
 $ is contained in
 the  \lclo-closed radical ideal generated by
 $\la ^2 + b  \la + a_2b   $, as seen by the Nullstellensatz (Theorem~\ref{Null3}) or by
 direct computation:
$$(\la ^2 + b  \la + a_1b )^2 \lmodL \left(\la ^2 + b \la + a  _2b
\right )\bigg(\la^2+ \frac {a_1b } {a_2} \la + \frac {a_1^2b
}{a_2}\bigg).
$$
\end{enumerate}
 \end{example}

But taking $\lmodL$-closed ideals usually is not enough for our
purposes, and we consider a more restrictive property in
Section~\ref{HB}.

\begin{rem} If the point $a$ is a tangible, isolated corner root of an essential
 polynomial
$f = \sum \a_i \la^{i},$
 then for some~$j$ we have $$a = \frac{\a_{j}}{\a_{j-1}},$$ with
$\a_{j-1},$ $\a_{j},$ and $\a_{ {j+1}}$ all tangible. Thus, $\la+a$
divides $f$, in view of
\cite[Proposition~8.40]{IzhakianRowen2007SuperTropical}. \end{rem}

 \begin{example}\label{gen3} The set $\tB = \{ h+ \al,\ h^{-1} +
 \xl{\beta}{\ell}  \}$ is redundant iff $\al^{-1} \le_\nu  \beta$.  The set
of binomials $$\{\la ^2 + \al\la = \la(\la + \al), \  \xl{\al}{\ell} \la +
\gm  = \la\gm  (\la ^{-1} + \gm ^{-1} \xl{\al}{\ell})\}$$  is redundant
iff $\al^{-1}\le _\nu \gm ^{-1} \al,$ i.e., $\gm  \le _\nu \al
^2$.
\end{example}

We call a binomial $h+h'$  \textbf{half-ghost} if $h$ is tangible
and $h'$ is ghost.

\begin{thm}\label{ngener} Suppose $F$ is 1-divisibly closed. For every radical
\bIdeal \,  $A$ of $ F[\Lm , \Lm ^{-1}]$ (for $\Lm  = \{ \la_1,
\dots, \la _n \}$), the set of binomials  of $A$ is generated
(also using the exchange property) by at most $2n$ irredundant
binomials, at most $n$ of which are tangibly spanned (with the
rest half-ghost).
\end{thm}
\begin{proof}  We start with Theorem~\ref{ngener0}, which gives us at most $n$  irredundant tangibly
spanned binomials. But  Example \ref{gen3} shows that when the
constant term is ghost, we might be able to adjoin a binomial
involving $h^{-1}.$ Applying this observation to Remark
\ref{binomgen}(iii) shows that any irredundant set of binomials of
$A$ has at most $2n$ elements.
\end{proof}

\section{Analogs of classical theorems from commutative algebra}\label{HB}

We turn now to the layer generation of corner ideals of the
polynomial semiring $\mathcal R := \Pol(F^{(n)},F)$. Two of the
cornerstones of ideal theory are the Principal Ideal Theorem, that
every ideal of $\mathcal R$ is principal, and Hilbert's Basis
Theorem, that every ideal of $\mathcal R$  is finitely generated.
We focus on   the tropical analogs. Let us commence with some
problematic examples, even in the standard supertropical case in
one indeterminate. Despite these examples, we will obtain positive
results when restricting our attention to   those ideals related
to tropicalization.

\begin{example}\label{quadout1} $ $ (The standard supertropical case)
\begin{enumerate}
\item Suppose $f _1 =  \la^2 + 5^\nu + 7 , $ $f _2 =
\la^2 + 4.9^\nu + 6.95,$ $f _3 =  \la^2 + 4.89^\nu + 6.945, \dots
$ all in $\mcR.$ The respective ghost loci have tangible parts
$[2,5] \supset [2.05, 4.9] \supset [2.055, 4.89] \supset \cdots$
which are decreasing, but, for each $i$, $f_i$ does not ghost
surpass $f_{i+1}.$ The ideal comprised of those polynomials whose
tangible corner locus is the intersection of these tangible
intervals, is not f.g. (Also, it is not an \aIdeal.) \pSkip

\item Likewise, take $f _i =  \la^2 + 5^\nu + a_i$ where $a_1 = 7
<_\nu a_2 <_\nu \dots  <_\nu 9.$ Again, the respective ghost loci
decrease, and if $a_i \to _\nu 9,$ then the $f_i$ generate a prime
ideal of $\mcR.$

\end{enumerate}
\end{example}

 In several indeterminates,
we can make Example~\ref{quadout1} even worse.

\begin{example}\label{quadout2} The ideal of Example~\ref{infgen0}(2)  is not   f.g.
Note that  $\tZ_{\tng}(f_k)$ is comprised of three rays, two being
the ``bent line'' $\tC$ comprised of rays to the left and beneath
$(0,0)$, and the third being a  ray in the upper right quadrant
whose slope depends on $k$. Hence, $\bigcap _i \tZ_{\corn}(f_i)$
is just $\tC$ (which is not a tropical curve in the usual sense).
\end{example}


\begin{example} The polynomials $f = (\la _1 +c)\la_2^2 + 100 \la_2 + 105$
for $c$ $\nu$-small all have tangible corner loci whose
intersection is given by $\bfa = (a_1, a_2)$ with $a_2 = 5,$ and
define an infinite ascending sequence of ideals.\end{example}

\subsection{Partial positive results involving geometric properties of ideals}

We can bypass these examples by imposing more stringent geometric
considerations. Here is some easy information garnered in the
standard supertropical case ($L = \{1,\infty\}$).

\begin{prop}\label{tang1} If $ a_1,
\dots, a_n $ are roots of $f\in F[\la]$ in distinct components,
then $$f \lmodL  (\la +a_1)\cdots (\la +a_n)h $$ for some $h\in
F[\la].$
\end{prop}
\begin{proof} As can be seen via
\cite[Theorem 8.41 and Proposition
8.47]{IzhakianRowen2007SuperTropical}, $f$ factors as $$f = (\la +
a_{i_1}^\nu)(\la^\nu + a_{i_2})\prod _{i \in I_1}(\la + a_i) \prod
_{i \in I_2}(\la^2 + b_i^\nu \la + b_i c_i)$$ where $c_i \le_\nu
a_i \le_\nu b_i;$ here $I_1$ indexes the ``corner roots'' of $f$
and $I_2$ indexes sets  of the ``cluster roots'' (other than $
a_{i_1}, a_{i_2}$).

  But by inspection   $$\la^2 + b_i^\nu \la + b_i c_i \lmodL (\la + a_i)
  \bigg(\la + \frac {b_ic_i}{a_i}\bigg),$$
so letting $h = \prod _{i \in I_2} (\la + \frac {b_ic_i}{a_i}),$
we have $$\prod _{i \in I_2}(\la^2 + b_i^\nu \la + b_i c_i) \lmodL
\prod _{i \in I_2}(\la +a_i) h   ,$$ implying $f \lmodL (\la
+a_1)\cdots (\la +a_n)h$.
\end{proof}

A  polynomial of the form $f = a_n\la^n + a_i^\nu \la^{n-1} +
\dots + a_1^\nu \la + a_0$ is \textbf{semitangibly-full} when it
has no inessential monomials \cite[Definition
8.29]{IzhakianRowen2007SuperTropical}.

\begin{lem}\label{tang11}  Suppose $f\in F[\la].$
If $[a,b]$ is a closed component of  $\tZ_{\tng}(f) $, then $\la^2
+ b^\nu + ab$ divides
 $f$. \end{lem}
\begin{proof}
 Applying \cite[Theorem 8.35]{IzhakianRowen2007SuperTropical}, in the terminology of
 \cite[Definition~8.33]{IzhakianRowen2007SuperTropical}, we can
factor $f$ into semitangibly-full  polynomials, one of whose
ghost loci contains the tangible interval $[a,b]$ and thus
assume that $f$ itself is semitangibly-full. An application of
\cite[Proposition~8.46]{IzhakianRowen2007SuperTropical} now
enables us to factor out quadratic factors from a
semitangibly-full polynomial.
\end{proof}

\subsection{Monomial-eliminating ideals}\label{Eucl}

  Proposition~\ref{tang1}
could be viewed as a version of the  principal ideal theorem, for
ideals defined in terms of a finite set of corner roots, but is
quite restrictive when viewed algebraically. We  prefer a more
intrinsically algebraic version which   can cope with the
counterexamples  given above.
 In this subsection, we further restrict the kinds of
 ideals, using natural algebraic properties naturally arising in tropical geometry, in order to be able to bypass these counterexamples and obtain a principal ideal theorem  parallel to the classical commutative Noetherian theory.

 The  notion    of   principal ideal is delicate even in the standard supertropical theory, since
for example, the ideal of $\Pol(F,F)$ consisting of polynomials for
which $1$ is a root contains both $\la + 1$ and $\la + 2^\nu$, and
thus is not principal in this strict sense. Here is the
layered version for  $\mcR = \Pol(F^{(n)},F)$.

   \begin{lem}\label{twosurp1} $g +\2q  \lmodWL g$ for any $g,q \in \mcR.$ \end{lem}
\begin{proof} Apply Lemma~\ref{twosurp} to each $\bfa \in S$. \end{proof}

\begin{defn} Suppose $F$ is a layered 1-\semifield0, and  $A \triangleleft \tR$.
The polynomial  $g \in \tR$ is tangibly $L$-\textbf{generated} by
tangibly spanned polynomials $f_1, \dots, f_m$ if  $ \sum_{i=1}^m
h_i f_i \lmodL  g$ for suitable tangibly spanned polynomials $h_1,
\dots,
 h_m$  satisfying the
following conditions:
\begin{enumerate}\eroman
\item  $ \sum_{i\ne j} h_i f_i \not \lmodL h_j f_j$ for each $1
\le j \le
 m$,
\item
 $\deg h_i f_i  \le \deg g$ with respect to the lexicographic order, for each $i$.
\end{enumerate}

An ideal $A$ is \textbf{tangibly $L$-principal}  iff  there is
some tangible $f\in A $ which tangibly $L$-generates every element
of $A$.
\end{defn}

\begin{example} (Here $\mathcal R =
\Pol(F,F).$)
\begin{enumerate}\eroman \item $\mathcal R(\la +1) + \mathcal R(\la +3)$ is a proper
ideal of $\mathcal R$ (containing all multiples of $(\la+a)$ with
$1 \nule a \nule 3$), which is not tangibly $L$-principal. \pSkip

\item  $A = \mcR(\la ^2 + 3\la +4) + \mcR(\la ^2   +4)$ is a proper exchange ideal of $\mathcal R$ which is not  tangibly $L$-principal, seen at once by comparing root
loci.
\end{enumerate}
\end{example}

 Furthermore, prime ideals could require an infinite number of generators;
consider, for example, the prime ideal generated by $\{\la + \a :
\a <_\nu 2 \}.$ (Of course, it is not corner.)

Since it is not enough to consider exchange ideals,
we also introduce an axiom motivated from Proposition~\ref{tropc}.

\begin{defn}\label{tropc1}  An ideal $A$ is called
\textbf{monomial-eliminating} if it has the following property:

Suppose $f,g \in A $ where $A $ is a tropicalized ideal. Then, for
any $h\in \supp(f)\cap \supp(g)$ we can write  $$f+g = \fgin  +
\fgot,$$ where   $ h \in \supp (\fgin ) \subseteq \supp(f) \cap
\supp(g),$ and $\fgot \in A $ with $ \supp(\fgot  ) \subseteq (
\supp(f) \cup \supp(g)) \setminus (\supp (\fgin )\cup \{ h \}) .$
\end{defn}

Monomial-eliminating ideals seem to provide  the proper
formulation for some of the classical results from commutative
ideal theory.

\subsubsection{The $L$-principal ideal theorem}\label{Eucl1}

\begin{thm}\label{Euclid} Over a layered 1-\semifield0 $F$, any   corner monomial-eliminating
 ideal   $A$ of $\mathcal R : = \Pol (F,F)$ is  tangibly $L$-principal with a
unique monic tangibly spanned generator, namely that tangibly
spanned monic polynomial $f \in A$ of minimal degree.
\end{thm}
 \begin{proof} By induction on degree, we claim that every  tangibly spanned polynomial $g \in A$ is  tangibly
$L$-generated by~$f$. Since $F$ is a 1-\semifield0, we may assume that $\deg f \le \deg g$ and $g$ is monic.
Letting $d := \deg f$ and $t := \deg g - d,$ we see that $\la^t f$
and $g$ have the same leading term $\la^{\deg g}$, so by
monomial-elimination, we can write $\la ^t  f + g  = \fgin +
\fgot,$ as in Definition~\ref{tropc1} (taking $\la^t f $ instead
of $f$), where $\fgot \in A$ and $\la^{\deg g} \notin \supp(\fgot
),$
 implying $\deg \fgot<\deg g.$ By Remark~\ref{Fun1} we can replace
 $\fgot $ by a tangible polynomial $\fgot ' \in A.$ Also, every monomial of $\fgin$ appears in both $\la ^t f$ and $g$.
 By induction $ \qq f \lmodL \fgot $  for some tangibly spanned polynomial~$\qq := \qq(\la)$ which must have
 degree at most $m-1$, 
 implying
\begin{equation} \label{checkit}  (\la^t + \qq)f  \lmodL    \la ^t f+ \fgot =  \fgin + \2\fgot   \lmodL g, \end{equation}
since any monomial of $\la^t f$ not in $g$ appears both in $\la^t
f$ and in~$    q f.$
 We conclude that $
(\la^t +\qq)f$ is the desired  $L$-factorization of $g$
 since $g$ has been presumed tangible.
\end{proof}

Alternatively, we could focus on prime exchange ideals, using the
same kind of proof.

\begin{thm}\label{Euclid1} Over a layered 1-\semifield0 $F$, any  prime     exchange ideal   $A$ of $\Pol (F,F)$
contains a tangibly spanned binomial $f$ such that each $g\in A$
contains a polynomial $L$-generated by $f$.
\end{thm}
 \begin{proof} By induction on degree, we claim that the assertion holds for every   $g \in A$
 of lower degree.  In view of Remark~\ref{Fun1}, we may assume that $g$ is tangibly spanned. Dividing out by its leading coefficient, we may assume that $g$ is monic. Then
letting $d := \deg f$ and $t := \deg g - d,$ we see that $\la^t f$
and $g$ have the same leading term $\la^{\deg g}$, so by the
exchange property, we can write $\la ^t f  =  h + p_1$ and $ g =
h+p_2 ,$ where~$h$ is the sum of those monomials of common
$\nu$-value in $\la ^t f $ and $ g $. By
Corollary~\ref{exchange12}, either $p_1+p_2 \in A$ or $h \in A$.
In the first case we apply induction on the degree. Thus, we may
assume that $h \in A$. Then we are done by induction on $|\supp
g|$ unless $\la ^t f $ is a layered factorization of $ g,$ which
is what we wanted to prove.
\end{proof}


\subsubsection{Finite tangible generation of ideals}\label{Hilb}

The same kind of approach works for finite generation.

\begin{defn} An ideal $A$ of $\mcR$ is \textbf{tangibly $L$-f.g.} if there is a finite (tangibly spanned) subset of $A$ that  tangibly $L$-generates each tangibly spanned polynomial of $A$.
A~\semiring0 is \textbf{tangibly $L$-Noetherian} if each ideal is tangibly $L$-f.g.
\end{defn}

Although this looks like a rather restrictive definition, we must
take into account the following example:

\begin{example} The polynomials $f = (\la _1 +c)\la_2^2 + 100 \la_2 + 105$ for $c$ $\nu$-small all have tangible root sets whose tangible
intersection is $\bfa = (a_1, a_2)$ with    $a_2 = 5,$ and
generate an ideal that is not tangibly $L$-f.g.\end{example}

Although prime corner ideals should play a special role, we
must cope with the following example pointed out to us   by
Sheiner.

\begin{example} The prime corner ideal $\tI_\corn  (( 0,0))$
contains $\la_1^i + \la_2^j$ for all $i,j$, and in particular
requires an infinite number of generators, as seen by taking
$j=1$.
\end{example}

 \begin{thm}\label{fg1} If   $F$ is a layered 1-\semifield0,   then the polynomial \semiring0  $\mcR:= \Pol(F^{(n)},F)$
 is tangibly $L$-Noetherian.
\end{thm}

\begin{proof}  We modify the usual proof of the Hilbert Basis Theorem, cf.~\cite{Row2006}. Let
$R :=  \Pol(F^{(n-1)},F).$ We need to show that any
monomial-eliminating ideal $A$ of $\mcR$, viewed as~$R[\la_n]$, is
tangibly $L$-f.g., by induction on $n$. We write $A_{m}$ for the
ideal of $R$ that is $L$-generated by the leading coefficients of
all polynomials of $A$ of degree $\le m $ in $\la_n$.

By induction on $n$, the ideal $\bigcup _{ m \in \Net} A_{m} $ of
$R$  is tangibly $L$-f.g. Taking $m'$ to be the maximal $m$
appearing for these finitely many tangibly spanned $L$-generators,
we explicitly write these
  $L$-generators as  $\al_{i,m}\in
A_m$, $1 \le m \le m',\ 1 \le i \le t_m$.
 We choose  tangibly spanned $f_{i,m} \in A$
of degree $m$, such that $\al_{i,m}$ is the leading coefficient
of~$f_{i,m}.$  (Note that for degree 0, $f_{i,0}=\al_{i,0}$.) Let
$ A' $ be the  ideal of $\mcR$ generated by $\{ f_{i,m}: 1 \le m
\le m', $ $1 \le i \le t_m\}$.

\pSkip \noindent \textbf{Claim:}  \emph{Every tangibly spanned $g
\in A$ lies in $A'$.}

 The claim is proved by induction on the lexicographic degree $\bold d = (d_1, \dots, d_n)$ of
 $g$ (as an $n$-tuple). If $d_n =0$ this is obvious, so assume
 $d_n>0$. We write $d$ for $d_n$.
  We shall lower the lexicographic degree
of~$f$ by means of the leading coefficients. Write $$g = \al\la
^{d} + \text{monomials  of lower  degree in $\la_n$}$$ (with
respect to the lexicographic degree in $R$).  Then $\al \in
A_{d}.$ \pSkip \noindent
\bfem{Case I.} \ $d \le m$. 
We take $  \sum _{i=1}^{t_d} r_{i}\al_{i,d}\lmodL  \al,$
  for suitable $r_{i}$ in~$R$, where $\deg  (r_{i}\al_{i,d})< \deg \al.$ Taking $f_{i,d}\in A$ of degree $d$ in $\la _n$ as above, we see that $  \sum _{i=1}^{t_d}
r_{i,d}f_{i,d}$ has the same leading term as $g$, so  some
$r_{i,d}f_{i,d}$ has the same leading term as $g$;  we can write
$r_{i,d}f_{i,d} +g = \fgin + \fgot,$ by monomial-elimination, and
continue as in the proof of Theorem~\ref{Euclid}.

Namely, $\fgot$ has lower lexicographic order than $g$,  so $\fgot
\in A'$ by induction, yielding $\bt_{j,k,\fgot}$ such that
$$\sum_{j,k} \bt_{j,k,\fgot} f_{j,k}\lmodL \fgot,$$ with each $\deg( \bt_{j,k,\fgot} f_{j,k})< \deg q.$ Then
$$\sum_{i=1}^{d} r_{i,d}f_{i,d} +\sum _{j,k}
\bt_{j,k,\fgot} f_{j,k}\lmodL g $$ since any monomial of $\fgin$
appears in both sums. We conclude by throwing out duplications.
\pSkip \noindent
 \bfem{Case II.}\ $d>m.$ Then
$A_d = A_m,$ so we proceed exactly as in Case I, using $m$ instead
of $d$, except this time taking $ \sum _{i=1}^{t_m}
r_{i,m}f_{i,m}\la^{d-m} $ instead of $  \sum _{i=1}^{t_d}
r_{i,d}f_{i,d}$.
\end{proof}



\end{document}

%% file: tropMacro.tex
\usepackage{epsfig,latexsym,amsfonts,amssymb,amsmath,amscd,graphics,epic}
\usepackage{amsfonts,amssymb,amsmath,amscd,amsthm}
\usepackage[mathscr]{eucal}
\usepackage{mathrsfs}
\usepackage{oldgerm,units}
\usepackage{wrapfig,epsfig}

\usepackage{ifthen}
\usepackage{mathbbol}

\usepackage{amsthm}
\usepackage[mathscr]{eucal}
\usepackage{mathrsfs}
\usepackage{mathbbol}
\usepackage{oldgerm,units}
\usepackage{wrapfig}

\newtheorem{theorem}{Theorem}[section]

\newtheorem{definition}[theorem]{Definition}

\newtheorem{example}[theorem]{Example}

\newcommand{\Real}{\mathbb R}

\newcommand{\Net}{\mathbb N}

\newcommand{\one}{\mathbb{1}}
\newcommand{\zero}{\mathbb{0}}

\newcommand{\got}[1]{\frak{#1}}

\newcommand{\trop}[1]{\mathcal{#1}}

\newcommand{\tB}{\trop{B}}
\newcommand{\tC}{\trop{C}}

\newcommand{\tG}{\trop{G}}

\newcommand{\tI}{\trop{I}}

\newcommand{\tL}{\trop{L}}

\newcommand{\tS}{\trop{S}}
\newcommand{\tT}{\trop{T}}
\newcommand{\tR}{\trop{R}}

\newcommand{\tZ}{\trop{Z}}

\newcommand{\al}{\alpha}
\newcommand{\bt}{\beta}
\newcommand{\gm}{\gamma}

\newcommand{\lm}{\lambda}
\newcommand{\Lm}{\Lambda}

    \ifx\proof\undefined
    \newenvironment{proof}{
    \smallskip
    \noindent\emph{Proof.}}{\hfill\(\Box\)
    \bigskip
    } \fi

\newcommand{\bfem}[1]{\textbf{\emph{#1}}}

\newcommand{\ifdef}[3]{\ifthenelse{\equal{#1}{true}}{#2}{#3}}

\numberwithin{equation}{subsection}
\numberwithin{equation}{section}